\documentclass[12pt]{amsart}

\usepackage[utf8]{inputenc}
\usepackage{fullpage, mathrsfs, mathtools, amssymb, tikz,bm, hyperref, amscd, scrextend, tikz-cd, verbatim, enumerate}
\usepackage[inline]{enumitem}
\usepackage{microtype}
\usepackage{amsbsy}
\usepackage{amsthm}

\usetikzlibrary{matrix, arrows}

\input xy
\xyoption{all} 

\newtheorem{theorem}{Theorem}[section]
\newtheorem{lemma}[theorem]{Lemma}
\newtheorem{prop}[theorem]{Proposition}
\newtheorem{cor}[theorem]{Corollary}

\newtheorem{definition}[theorem]{Definition}

\newtheorem{remark}[theorem]{Remark}

\newcommand{\Q}{{\mathbb Q}}
\newcommand{\R}{{\mathbb R}}
\newcommand{\Z}{{\mathbb Z}}

\newcommand{\LL}{\pmb{\mathscr L}}

\newcommand{\PP}{{\mathbb P}}

\newcommand{\calO}{\mathcal O}
\newcommand{\calM}{\mathcal M}

\newcommand{\calA}{\mathcal A}

\newcommand{\Supp}{\mathrm{Supp}}
\newcommand{\Spec}{\mathrm{Spec}}

\newcommand{\mb}[1]{\mathbb{#1}}

\usepackage{enumitem}

\title{Log canonical models of elliptic surfaces}
\author{Kenneth Ascher \& Dori Bejleri}

\begin{document}
\maketitle

\begin{abstract}We classify the log canonical models of elliptic surface pairs $(f: X \to C,S + F_\calA)$ where $f: X \to C$ is an elliptic fibration, $S$ is a section, and $F_\calA$ is a weighted sum of marked fibers.  In particular, we show how the log canonical models depend on the choice of the weights. We describe a wall and chamber decomposition of the space of weights based on how the log canonical model changes. In addition, we give a generalized formula for the canonical bundle of an elliptic surface with section and marked fibers. This is the first step in constructing compactifcations of moduli spaces of elliptic surfaces using the minimal model program. \end{abstract}

\section{Introduction}
One goal of birational geometry is to find distinguished birational models for algebraic varieties-- in dimension one these are the unique smooth projective models, but in higher dimensions we are led to minimal and canonical models which may have mild singularities. More generally, the log minimal model program takes as input a pair $(X,D)$ consisting of a variety and a divisor with mild singularities and outputs a log minimal or log canonical model of the pair. Log canonical models and their non-normal analogues, semi-log canonical (slc) models, are the higher dimensional generalization of stable curves and lend themselves to admitting compact moduli spaces. 

Inspired by La Nave's explicit stable reduction of elliptic surface pairs $(X,S)$ in \cite{ln}, we give a classification of the log canonical models of elliptic surface pairs $(f: X \to C, S + F_{\calA})$, where $f: X \to C$ is an elliptic fibration, $S$ is a chosen section, and $F_{\calA}$ is a weighted sum of reduced marked fibers $F_\calA = \sum a_i F_i$ (see Definition \ref{slcweighted}). We define an elliptic fibration as a surjective proper flat morphism $f: X \to C$ from an irreducible surface $X$ to a proper smooth curve $C$ with section $S$ such that the generic fiber of $f$ is a stable elliptic curve. We obtain a complete description of the log canonical models building on the classification of singular fibers of minimal elliptic surfaces given by Kodaira and Ner\'on. Our aim is to explicitly describe how the log canonical models of elliptic surface pairs depend on the choice of the weight vector $\calA$. Drawing inspiration from the Hasset-Keel program for $\overline{\calM}_{g,n}$, we will use these results to understand how the geometry of compactified moduli spaces of slc elliptic surface pairs vary as we change the weight vector $\calA$ \cite{master}. \\
\indent Our first main result is the following classification (see Figure 1):

\begin{theorem}\label{thm:thm1} Let $(f: X \to C, S + aF)$ be an elliptic surface pair over $C$ the spectrum of a DVR with reduced special fiber $F$ such that  $F$ is one of the Kodaira singular fiber types (see Table 1), or  $f$ is isotrivial with constant $j$-invariant $\infty$. 

\begin{enumerate} 
\item If $F$ is a type $\mathrm{I}_n$ or $\mathrm{N}_0$ fiber (see Definition \ref{inftyfibers}), the relative log canonical model is the Weierstrass model (see Definition \ref{def:weierstrass})  for all $0 \le a \le 1$.
\item For any other fiber type, there is an $a_0$ such that the relative log canonical model is
\begin{enumerate}[label = (\roman*)]
\item the Weierstrass model for any $0 \le a \le a_0$,
\item a \emph{twisted fiber} (see Defintion \ref{def:twisted1}) consisting of a single non-reduced component when $a = 1$, or
\item an \emph{intermediate fiber} (Definition \ref{def:twisted1}) that interpolates between the above two models for any $a_0 < a < 1$.
\end{enumerate}
\end{enumerate}

The constant $a_0 = 0$ for fibers of type $\mathrm{I}_n^*, \mathrm{II}^*, \mathrm{III}^*$ and $\mathrm{IV}^*$, and $a_0$ is as follows for the other fiber types:
$$
a_0 = \left\{ \begin{array}{lr} 5/6 & \mathrm{II} \\ 3/4 & \mathrm{III} \\ 2/3 & \mathrm{IV} \\  1/2 & \mathrm{N}_1 \end{array}\right.
$$
\end{theorem}

We also describe the singularities of the relative log canonical models in each case. 

 \begin{center}
\begin{figure}[!h] \caption{\footnotesize{Transitions between (left to right): Weierstrass, intermediate, and twisted fibers.}}
\centering\includegraphics{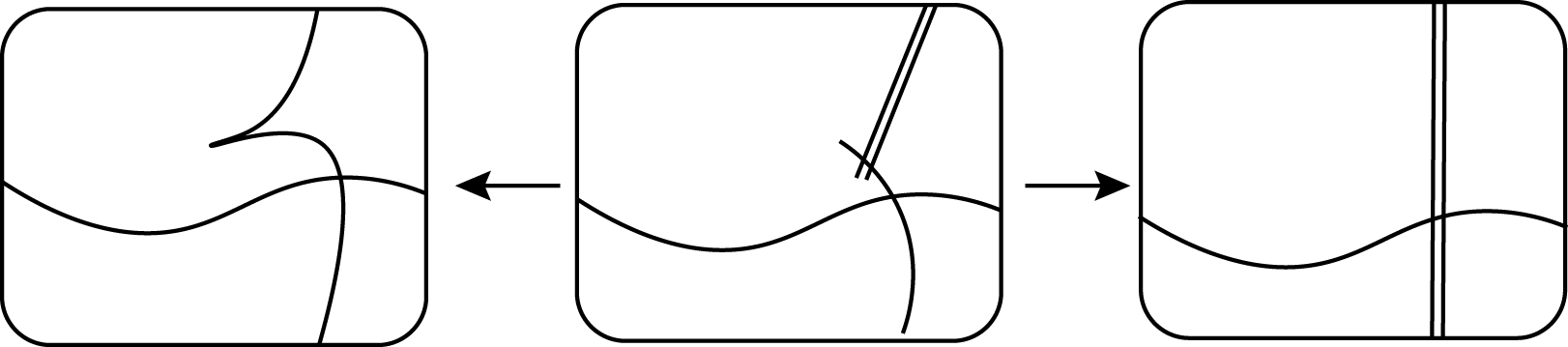} \end{figure}
\end{center}

Theorem \ref{thm:thm1} allows us to run the log minimal model program for $(X, S + F_\calA)$ relative to the map $f : X \to C$ to produce a relative log canonical (or \emph{relatively stable}) model over the curve $C$. Indeed this question is local on the target so it reduces to the case $(X, S + aF)$ where $f : X \to C$ is an elliptic fibration over the spectrum of a DVR and $F$ is the reduced special fiber. When the generic fiber of $f$ is smooth, $F$ is one of the singular fibers in Kodaira's classification (see Table 1). When the generic fiber of $f$ is a nodal elliptic curve, the fibration $f$ must be isotrivial with constant $j$-invariant $\infty$ and we classify the singular fibers by explicitly using their Weierstrass models (see Lemmas \ref{slckodaira} and \ref{inftyfibers}). 

In \cite{ln}, La Nave studied degenerations of \emph{Weierstrass elliptic surfaces}. The approach was to replace any cuspidal fibers with twisted fibers, study degenerations using twisted stable maps of Abramovich-Vistoli, and then reinsert cuspidal fibers to obtain Weierstrass models. Theorem \ref{thm:thm1} can be seen as a generalization of his gluing procedure, which shows that instead the log minimal model program naturally interpolates between the Weierstrass and twisted fibers.

Using the local computation of relative log canonical models we generalize the elliptic surface canonical bundle formula (see Proposition \ref{classcan}) to elliptic surface pairs. 

\begin{theorem}[see Theorem \ref{canonical}] Let $f: X \to C$ be an elliptic fibration with section $S$. Furthermore, let $F_{\calA} = \sum a_i F_i$ be a sum of reduced marked fibers $F_i$ with $0 \le a_i \le 1$. Suppose that $(X, S + F_{\calA})$ is the relative semi-log canonical model over $C$. Then 
$$
\omega_X = f^*(\omega_C \otimes \LL) \otimes \calO_X(\Delta).
$$
where $\LL$ is the fundamental line bundle (see Definition \ref{linebundle}) and $\Delta$ is an effective divisor supported on fibers of type $\mathrm{II}, \mathrm{III}$, and $\mathrm{IV}$ contained in $\Supp(F)$. The contribution of a type $\mathrm{II}$, $\mathrm{III}$ or $\mathrm{IV}$ fiber to $\Delta$ is given by $\alpha E$ where $E$ supports the unique nonreduced component of the fiber and
$$
\alpha = \left\{ \begin{array}{lr} 4 & \mathrm{II} \\ 2 & \mathrm{III} \\ 1 & \mathrm{IV} \end{array} \right. 
$$
 \end{theorem}

In continuing the log minimal model program on the relatively stable pair $(X \to C, S + F_{\calA})$, sometimes the section $S$ is contracted. This was first observed by La Nave when studying compactifications of the space of Weirstrass models (i.e. $a_i = 0$ for all $i$) \cite{ln}. La Nave called the result of such a contraction a \emph{pseudoelliptic surface} (see Definition \ref{def:pseudo}). \\

In Proposition \ref{prop:adjunction} we compute the formula
$$
(K_X + S + F_\calA).S = 2g - 2 + \sum a_i
$$
using \cite[Proposition 4.3.2]{ln} where $g = g(C)$ is the genus of the base curve. The section gets contracted by the log minimal model program precisely when $(K_X + S + F_\calA).S \le 0$. 

It follows that the section does \emph{not} get contracted if and only if the base curve is a Hassett \emph{weighted stable pointed curve} \cite{has} (Definition \ref{hassettcurve}) with marked points $\sum a_i p_i$ where $p_i = f_*F_i$. In particular, the log minimal model program results in a pseudoelliptic surface  only when $C \cong \PP^1$ and $\sum a_i \le 2$, or when $C$ is an elliptic curve and $a_i = 0$ for all $i$. 

\begin{cor} Let $(f: X \to C, S + F_\calA)$ be an irreducible slc elliptic surface with section $S$ and marked fibers $F_\calA$. Suppose that $K_X + S + F_\calA$ is big. Then the log canonical model of $(X, S + F_\calA)$ is either 
\begin{enumerate}[label = (\alph*)]
\item the relative log canonical model as described in Theorem \ref{thm:thm1}, or
\item a pseudoelliptic surface obtained by contracting the section of the relative log canonical model whenever $(C, f_*F_\calA)$ is not a weighted stable pointed curve (see Definition \ref{hassettcurve}).
\end{enumerate}
\end{cor} 

When $K_X + S + F_\calA$ is \emph{not} big, there is a log canonical contraction mapping the surface to a curve or point. This is described in Section \ref{sec:genus0}, with respect to a classification based on $\LL$.\\

In Section \ref{sec:calculationsection}, we present a wall-and-chamber decomposition of the space of weight vectors $\calA$. The log canonical model of $(X \to C, S + F_\calA)$ remains the same within each chamber and we describe how it changes across each wall. Through an explicit example,  we demonstrate a rational elliptic surface pair that exhibits each of these transitions (see Example \ref{example}). \\

We use the results of this paper to study compact moduli spaces parametrizing $\calA$-weighted stable elliptic surfaces in \cite{master}, and describe wall crossing behavior as the weight vector $\calA$ varies. Finally, we note that in \cite{tsm}, we give an alternative viewpoint for both the moduli spaces and surfaces when $\calA = 1$, via Abramovich-Vistoli twisted stable maps \cite{av}.\\

We work over an algebraically closed field $k$ of characteristic 0.

\subsection*{Acknowledgements} We thank our advisor Dan Abramovich for his constant guidance and support, and the referee for their suggestions on improving the exposition of this paper. Research of both authors supported in part by funds from NSF grant DMS-1500525. K.A. was partially supported by an NSF Postdoctoral Fellowship. 

\section{The log minimal model program in dimension two}
First we recall some facts about the log minimal model program that we will use later. The standard reference is \cite{km} (for generalizations to log canonical surface pairs see \cite{fujinosurface}).

Throughout this section, $X$ will denote a connected surface, $D = \sum a_iD_i$ will be a $\Q$-divisor with $0 \le a_i \le 1$, and $(X,D)$ will be referred to as a surface pair.

 \begin{definition}\label{def:loc} Let $(X, D)$ be a  surface pair such that $X$ is normal and $K_X + D$ is a $\Q$-Cartier. Suppose that there is a log resolution $f: Y \to X$ such that $$K_Y + \sum a_E E = f^*(K_X + D),$$ where the sum goes over all irreducible divisors on $Y$. We say that the pair $(X,D)$ has \textbf{log canonical singularities} (or is lc) if all $a_E \leq 1$. \end{definition}

We let $NS(X)$ denote the $\Q$-vector space of $\Q$-divisors modulo numerical equivalence. If $f : X \to S$ is a projective morphism, denote by $N_1(X/S)$ the $\Q$-vector space generated by irreducible curves $C \subset X$ contracted by $f$ modulo numerical equivalence.

Let $\overline{NE}(X/S) \subset N_1(X/S)$ be the closure of the cone generated by effective curve classes. For any divisor $D \in NS(X)$, we let 
$$
\overline{NE}(X/S)_{D \geq 0} = \overline{NE}(X/S) \cap \{C : D.C \geq 0\}.
$$
The first step of the log minimal model program is to understand these cones:

\begin{theorem} (Cone and contraction theorems for log canonical surfaces) Let $(X,D)$ be a log canonical surface pair and $f : X \to S$ a projective morphism. 

\begin{enumerate}[label = (\alph*)] 
\item There exist countably many rational curves $C_j \subset X$ contracted by $f$ and
$$
\overline{NE}(X/S) = \overline{NE}(X/S)_{(K_X + D) \geq 0} + \sum_j \R_{\geq 0}[C_j]
$$
such that $R_j := \R_{\geq 0}[C_j]$ is an extremal ray for each $j$. That is, $R_j$ satisfies $x,y \in R$ whenever $x + y \in R$ for any curve classes $x,y$.
\item For each extremal ray $R$ as above, there exists a unique morphism $\phi_R : X \to Y$ such that $(\phi_R)_*\calO_X = \calO_Y$ and $\phi_R(C) = 0$ for an integral curve $C$ if and only if $[C] \in R$. In particular, $f : X \to S$ factors as $g \circ \phi_R$ for a unique $g : Y \to S$. The pair $(Y, (\phi_R)_*D)$ is a log canonical surface pair. 
\end{enumerate}  \end{theorem}

The morphism $\phi_R$ is called an \emph{extremal contraction}. The log minimal model program takes as input a pair $(X,D)$ and applies the above theorem repeatedly to construct a sequence of extremal contractions in hopes of reaching a birational model $f_0 : (X_0,D_0) \to S$ so that $K_{X_0} + D_0$ is $f_0$-nef. That is, with
$
\overline{NE}(X_0/S) = \overline{NE}(X_0/S)_{(K_{X_0} + D_0) \geq 0}. $\\
 \indent If it exists, the pair $(X_0,D_0)$ is a log minimal model of $(X,D)$ over $S$. Note that this pair is not unique and depends on the sequence of extremal contractions. For surfaces, it can be shown (e.g. using the Picard rank) that a sequence of extremal contractions leads to a log minimal model after finitely many steps. Once we reach a log minimal model, the next step is provided by the abundance theorem: 

\begin{prop}\cite{afkm,kawamata}\label{prop:abundance}(Abundance theorem for log canonical surfaces) Let $(X,D)$ be a log canonical surface pair and $f : X \to S$ a projective morphism. If $K_X + D$ is $f$-nef, then it is $f$-semiample. \end{prop}

A divisor $B$ is \emph{$f$-semiample} if the linear series $|mB|$ induces a morphism
$
\phi_{|mB|} : X \to \mb{P}^N_S
$
over $S$ for $m \gg 0$. In this case, $\phi_{|mB|}$ is a projective morphism with connected fibers satisfying $\phi_{|mB|}^*H = mB$ where $H$ is the hyperplane class on $\mb{P}^N_S$. In particular, $\phi_{|mB|}(C) = 0$ if and only if $B.C = 0$ so $\phi_{|mB|}$ is determined by numerical data of the divisor  $B$. 

In the setting of the abundance theorem, the map $\phi_{|m(K_X + D)|}$ is the log canonical contraction and the image $(Y, (\phi_{|m(K_X + D)|})_*D)$ is the log canonical model of the pair $(X,D)$ over $S$. Thus
$$
Y = \mathrm{Proj}_S\left(\bigoplus_{n \geq 0} H^0(X, n(K_X + D)) \right).
$$ 

More generally, if $(X,D)$ is any pair projective over $S$, the log canonical model of $(X,D)$ over $S$ is the pair $(Y, \phi_*D)$ where
$$
\phi : X \dashrightarrow Y = \mathrm{Proj}_S\left(\bigoplus_{n \geq 0} H^0(X, n(K_X + D)) \right)
$$
is the rational map induced by the linear series $|m(K_X + D)|$ for $m \gg 0$. Note that the log canonical model is unique and $K_Y + \phi_*D$ is relatively ample over $S$. 

The log minimal model program then gives us a method to compute the unique log canonical model of an lc surface pair $(X,D)$ projective over a base $S$. First use the cone and contraction theorems to perform finitely many extremal  contractions of $K_X + D$-negative curves to obtain a log minimal model $(X_0, D_0)$ over $S$. Then $K_{X_0} + D_0$ will be relatively semi-ample by the abundance theorem so that the log canonical model is the image of the log canonical contraction 
$
\phi_0 : X_0 \to Y
$
that contracts precisely the curves $C$ such that $(K_{X_0} + D_0).C = 0$.

\subsection{SLC pairs and stable pairs} 

To obtain compact moduli spaces, one is naturally forced to allow non-normal singularities as in the case of stable curves. The non-normal analogue of log canonical pairs is the following: 
	
	\begin{definition}\label{def:slc} Let $(X, D)$ be a pair of a reduced variety and a $\Q$-divisor such that $K_X + D$ is $\Q$-Cartier. The pair $(X,D)$ has \textbf{semi-log canonical singularities} (or is slc) if:
		\begin{enumerate}
			\item The variety $X$ is S2,
			\item $X$ has only double normal crossings in codimension 1, and 
			\item If $\nu: X^{\nu} \to X$ is the normalization, then the pair $(X^{\nu}, \sum d_i \nu_*^{-1}(D_i) + D^{\nu})$ is log canonical, where $D^{\nu}$ denotes the preimage of the double locus on $X^{\nu}$. \end{enumerate}
	\end{definition}
	
	\begin{definition}
	A surface is \textbf{semi-smooth} if it only has the following singularities:\begin{enumerate}
		\item 2-fold normal crossings (locally $x^2 = y^2$), or
		\item pinch points (locally $x^2 = zy^2$).
	\end{enumerate}
\end{definition}

This naturally leads to the definition of a \emph{semi-resolution}.

\begin{definition}
	A \textbf{semi-resolution} of a surface $X$ is a proper map $g: Y \to X$ such that $Y$ is semi-smooth and $g$ is an isomorphism over the semi-smooth locus of $X$. A semi-log resolution of a pair $(X,D)$ is a semi-resolution $g : Y \to X$ such that $g_*^{-1}D + \mathrm{Exc}(g)$ is normal crossings. 
\end{definition}

\begin{remark} As in the case of a log canonical pair, the definition of a semi-log canonical pair can be rephrased in terms of a semi-log resolution (see Lemma \ref{logcanonical}).\end{remark} 
	\begin{definition}  A pair $(X, D)$ of a projective variety and $\Q$-divisor is a \textbf{stable pair} if:
\begin{enumerate} \item $(X, D)$ is an slc pair, and \item $\omega_X(D)$ is ample. \end{enumerate} \end{definition}

Note in particular that log canonical models are stable pairs and often we will use the words log canonical model and stable model interchangeably. 

The cone, contraction, and abundance theorems for log surfaces hold as stated above for slc surfaces (see for example \cite[Theorem 1.19]{fundamental} for the cone and contraction theorem and \cite[Theorem 1.4]{hx1} \cite[Theorem 1.4]{fg} for abundance). Thus one hopes that one can run the log minimal model program as described above to obtain the stable model of an slc pair $(X,D)$. However, this is not always the case. An extremal contraction of an slc pair might result in a pair $(X,D)$ where $K_X + D$ is \emph{not} $\Q$-Cartier! This becomes an important issue in moduli theory. However, this issue does not occur for the specific non-normal elliptic surfaces we consider in this paper.

\section{Preliminaries on elliptic surfaces} 

In this section we summarize the basic theory of elliptic surfaces.

\subsection{Standard elliptic surfaces} We point the reader to \cite{mir3} for a detailed exposition.

\begin{definition}\label{def:standardes} An irreducible \textbf{elliptic surface with section} $(f : X \to C, S)$ is an irreducible surface $X$ together with a surjective proper flat morphism $f: X \to C$ to a smooth curve and a section $S$ such that:
\begin{enumerate}
\item the generic fiber of $f$ is a stable elliptic curve, and
\item the generic point of the section is contained in the smooth locus of $f$. 
\end{enumerate}
We call $(f: X \to C, S)$  \textbf{standard} if $S$ is contained in the smooth locus of $f$. 

\end{definition} 

This differs from the usual definition in that we only require the generic fiber to be a \emph{stable} elliptic curve rather than a smooth one. 

\begin{definition}
	A surface is called \textbf{relatively minimal} if it semi-smooth and there are no $(-1)$-curves in any fiber. 
\end{definition}

\noindent Note that a relatively minimal elliptic surface with section must be standard. If $(f:X \to C, S)$ is relatively minimal, then there are finitely many fiber components not intersecting the section. Contracting these, we obtain an elliptic surface with all fibers reduced and irreducible: 

\begin{definition} \label{def:weierstrass} A \textbf{minimal Weierstrass fibration} is an elliptic surface obtained from a relatively minimal elliptic surface $(f : X \to C,S)$ by contracting all fiber components not meeting $S$. We refer to this as the \textbf{minimal Weierstrass model} of $(f : X \to C, S)$.  \end{definition}

\begin{definition}\label{linebundle} The \textbf{fundamental line bundle} of a standard elliptic surface $(f : X \to C, S)$ is $\LL := (f_*N_{S/X})^{-1}$ where $N_{S/X}$ denotes the normal bundle of $S$ in $X$. For $(f : X \to C, S)$ an arbitrary elliptic surface, we define $\LL := (f'_*N_{S'/X'})^{-1}$ where $(f' : X' \to C, S')$ is a semi-resolution. \end{definition}

Since $N_{S/X}$ only depends on a neighborhood of $S$ in $X$, the line bundle $\LL$ is invariant under taking a semi-resolution or the Weierstrass model of a standard elliptic surface. Therefore $\LL$ is well defined and equal to $(f'_*N_{S'/X'})^{-1}$ for $(f' : X' \to C, S')$ the unique minimal Weierstrass model of $(f : X \to C, S)$. 

The fundamental line bundle greatly influences the geometry of a minimal Weierstrass fibration. The line bundle $\LL$ has non-negative degree on $C$ and is independent of choice of section $S$ \cite{mir3}. Furthermore, $\LL$ determines the canonical bundle of $X$:

\begin{prop}\cite[Proposition III.1.1]{mir3}\label{classcan} Let $(f : X \to C, S)$ be either \begin{enumerate*} \item a Weierstrass fibration, or \item a relatively minimal smooth elliptic surface \end{enumerate*}. Then
$
\omega_X = f^*(\omega_C \otimes \LL).
$
\end{prop}

We prove a more general canonical bundle formula in Theorem \ref{canonical}.

\begin{definition} We say that $f : X \to C$ is \textbf{properly elliptic} if $\deg(\omega_C \otimes \LL) > 0$. \end{definition} 

It is clear that $X$ is properly elliptic if and only if the Kodaira dimension $\kappa(X) = 1$. \\

When $(f : X \to C, S)$ is a smooth relatively minimal elliptic surface, then $f$ has finitely many singular fibers. These are unions of rational curves with possibly non-reduced components whose dual graphs are $ADE$ Dynkin diagrams. The possible singular fibers were classified independently by Kodaira and Ner\'on. Table \ref{table} gives the full classification in Kodaira's notation for the fiber.  Fiber types $\mathrm{I}_n$ for $n \geq 1$ are reduced and normal crossings, fibers of type $\mathrm{I}^*_n, \mathrm{II}^*, \mathrm{III}^*$, and $\mathrm{IV}^*$ are normal crossings but nonreduced, and fibers  of type $\mathrm{II}, \mathrm{III}$ and $\mathrm{IV}$ are reduced but not normal crossings. With the exception of type $\mathrm{I}_0, \mathrm{I}_1$ and $\mathrm{II}$, all irreducible components of the fibers are $-2$ curves.

\begin{definition} We will use \textbf{reduced fiber} to mean the reduced divisor $F = (f^*(p))^{red}$ corresponding to a (possibly nonreduced) fiber $f^*(p)$ for $p \in C$.\end{definition} 

\begin{table}\label{table:kodaira}
\centering
\caption{\footnotesize{Kodaira's classification of fibers. The numbers in the Dynkin diagrams represent  multiplicities of  nonreduced components. Pictures shamelessly taken from Wikipedia.}    }
\label{table}
\centering
\begin{tabular}{|c|c|c|c|}
\hline
Kodaira Type        & \# of components           & Fiber                                                                             & Dynkin Diagrams                     \\ \hline
$\mathrm{I}_0$               & 1                          & \begin{minipage}{.3\textwidth}\includegraphics[scale=.5]{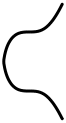}

\end{minipage} &  
\begin{minipage}{.3\textwidth}\begin{tikzpicture}
\draw[fill=black]
(0,0)                         
      circle [radius=.1] node [above] {}

; 

\end{tikzpicture}
\end{minipage}

 \\ \hline

$\mathrm{I}_1$               & 1 (double pt)              & \begin{minipage}{.3\textwidth}\includegraphics[scale=.5]{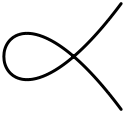}\end{minipage} & \begin{minipage}{.3\textwidth}\begin{tikzpicture}
\draw[fill=black]
(0,0)                         
      circle [radius=.1] node [above] {}

; 

\end{tikzpicture}\end{minipage}  \\ \hline

$\mathrm{I}_n$, $n \geq 2$   & $n$ ($n$ intersection pts) & \begin{minipage}{.3\textwidth}\includegraphics[scale=.5]{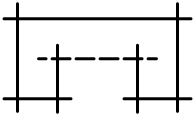}\end{minipage} & \begin{minipage}{.3\textwidth}
\begin{tikzpicture}
\draw
(0,0)  circle [radius=.1] 
(1,0) circle [radius=.1]
(3,0) circle [radius=.1] (4,0) circle [radius=.1]  
(0.1,0) -- (.9,0);
\draw[dashed]
(1.1,0) -- (2.9,0);
\draw[fill=black]
(2,.5) circle [radius=.1] 
(3.1,0) -- (3.9,0)
(2.15,.45) -- (3.9,.1)
(0.1,.1) -- (1.85,.45);
\end{tikzpicture}

\end{minipage}  \\ \hline
$\mathrm{II}$                & 1 (cusp)                   & \begin{minipage}{.3\textwidth}\includegraphics[scale=.5]{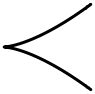}\end{minipage} & \begin{minipage}{.3\textwidth}\begin{tikzpicture}
\draw[fill=black]
(0,0)                         
      circle [radius=.1] node [above] {}

; 

\end{tikzpicture}

\end{minipage}    \\ \hline
$\mathrm{III}$               & 2 (meet at double pt)      & \begin{minipage}{.3\textwidth}\includegraphics[scale=.5]{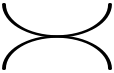}\end{minipage}   & \begin{minipage}{.3\textwidth}\begin{tikzpicture}
\draw[fill=black]
(0,0) circle [radius=.1];
\draw 
(1,0) circle [radius=.1];

\draw 
(.1,0) -- (.9,0);

\end{tikzpicture}

\end{minipage}\\ \hline
$\mathrm{IV}$                & 3 (meet at 1 pt)           & \begin{minipage}{.3\textwidth}\includegraphics[scale=.5]{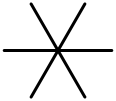}\end{minipage}  & \begin{minipage}{.3\textwidth}
\begin{tikzpicture}

\draw
(0,0) circle [radius =.1]
(1,0) circle [radius =.1] 
(.1,0) -- (.9,0)
(.1,.1) -- (.4,.4)
(.6,.4) -- (.9,.1)
;
\draw[fill=black]
(.5,.5) circle [radius =.1];
\end{tikzpicture}

\end{minipage}  \\ \hline
$\mathrm{I}_0^*$             & 5                          & \begin{minipage}{.3\textwidth}\includegraphics[scale=.5]{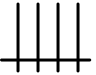}\end{minipage}& \begin{minipage}{.3\textwidth}\begin{tikzpicture}

\draw[fill=black]
(-1,-.5) circle [radius=.1];
\draw
(-1,.5) circle [radius=.1]  
(0,0) circle [radius=.1] node [above] {\tiny 2}
(-.85,-.45) -- (-.15,-.1)
(-.85,.45) -- (-.15,.1)
(1,.5) circle [radius=.1] 
(1,-.5) circle [radius=.1] 
(.15, .1) -- (.85,.45)
(.15,-.1) -- (.85,-.45)
;

\end{tikzpicture}

\end{minipage}   \\ \hline
$\mathrm{I}_n^*$, $n \geq 1$ & $5 + n$                    & \begin{minipage}{.3\textwidth}\includegraphics[scale=.5]{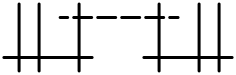}\end{minipage}    & \begin{minipage}{.3\textwidth}\begin{tikzpicture}
\draw[fill=black]
(-1,-.5) circle [radius=.1];
\draw
(-1,.5) circle [radius=.1] 
(0,0) circle [radius=.1] node [above] {\tiny 2}
(1,0)  circle [radius=.1] node [above] {\tiny 2}
(2,0)  circle [radius=.1] node [above] {\tiny 2};
\draw[dashed]
(.15,0) -- (.85,0)
(1.15,0) -- (1.9,0);
\draw
(-.85,-.45) -- (-.15,-.1)
(-.85,.45) -- (-.15,.1)
(3,.5) circle [radius=.1] 
(3,-.5) circle [radius=.1] 
(2.15, .1) -- (2.85,.45)
(2.15,-.1) -- (2.85,-.45)
;

\end{tikzpicture}

\end{minipage}\\ \hline
$\mathrm{II}^*$              & 9                          & \begin{minipage}{.3\textwidth}\includegraphics[scale=.5]{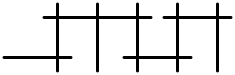}\end{minipage}   & \begin{minipage}{.3\textwidth}

\begin{tikzpicture}
\draw
(0,0) circle [radius=.1] node [below] {\tiny 2}
(.5,0)  circle [radius=.1] node [below] {\tiny 4}
(1,0)  circle [radius=.1] node [below] {\tiny 6}
(1,.5) circle [radius=.1] node [right] {\tiny 3}
(1.5,0)  circle [radius=.1] node [below] {\tiny 5}
(2,0)  circle [radius=.1] node [below] {\tiny 4}
(2.5,0)  circle [radius=.1] node [below] {\tiny 3}
(3,0)  circle [radius=.1] node [below] {\tiny 2}
(1,.1) -- (1,.4)
(.1,0) -- (.4,0)
(.6,0) -- (.9,0)
(1.1,0) -- (1.4,0)
(1.6,0) -- (1.9,0)
(2.1,0) -- (2.4,0)
(2.6,0) -- (2.9,0)
(3.1,0) -- (3.4,0);
\draw[fill=black]
(3.5,0) circle [radius=.1];
\end{tikzpicture}

\end{minipage}  \\ \hline
$\mathrm{III}^*$             & 8                          & \begin{minipage}{.3\textwidth}\includegraphics[scale=.5]{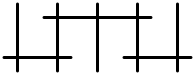}\end{minipage}   & \begin{minipage}{.3\textwidth}

\begin{tikzpicture}
\draw
(0,0) circle [radius=.1] 
(.5,0)  circle [radius=.1] node [below] {\tiny 2}
(1,0)  circle [radius=.1] node [below] {\tiny 3}
(1.5,0)  circle [radius=.1] node [below] {\tiny 4}
(1.5,.5) circle [radius=.1] node [right] {\tiny 2}
(2,0)  circle [radius=.1] node [below] {\tiny 3}
(2.5,0)  circle [radius=.1] node [below] {\tiny 2}
(1.5,.1) -- (1.5,.4)
(.1,0) -- (.4,0)
(.6,0) -- (.9,0)
(1.1,0) -- (1.4,0)
(1.6,0) -- (1.9,0)
(2.1,0) -- (2.4,0)
(2.6,0) -- (2.9,0);
\draw[fill=black]
(3,0) circle [radius=.1];
\end{tikzpicture}

\end{minipage}  \\ \hline
$\mathrm{IV}^*$              & 7                          & \begin{minipage}{.3\textwidth}\includegraphics[scale=.5]{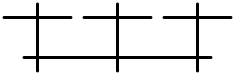}\end{minipage} & \begin{minipage}{.3\textwidth}

\begin{tikzpicture}
\draw
(0,0) circle [radius=.1] 
(.5,0)  circle [radius=.1] node [below] {\tiny 2}
(1,0)  circle [radius=.1] node [below] {\tiny 3}
(1.5,0)  circle [radius=.1] node [below] {\tiny 2}
(1.0,.5) circle [radius=.1] node [right] {\tiny 2}
(1.0,1) circle [radius=.1] 
(1.0,.1) -- (1.0,.4)
(1.0,.6) -- (1.0, .9)
(.1,0) -- (.4,0)
(.6,0) -- (.9,0)
(1.1,0) -- (1.4,0)
(1.6,0) -- (1.9,0);
\draw[fill=black]
(2,0)  circle [radius=.1];
\end{tikzpicture}

\end{minipage}  \\ \hline
\end{tabular}
\end{table}

\subsection{Weighted stable elliptic surfaces} 
		Following the log minimal model program, in \cite{master}  we will study compactifications of the moduli space of irreducible elliptic surfaces with section and marked fibers obtained by allowing our surface pairs to degenerate to semi-log canonical (slc) stable pairs.
			
		Let $\calA = (a_1, \dots, a_n) \in \Q^n$ such that $0 < a_i \leq 1$ for all $i$ be a \emph{weight vector}. We consider elliptic surfaces marked by an $\calA$-weighted sum 
		$
		F_\calA = \sum_{i = 1}^n a_i F_i
		$
		\noindent with $F_i$ reduced fibers.
		
		\begin{definition}\label{slcweighted} An \textbf{$\calA$-weighted slc elliptic surface} is a tuple $(f : X \to C, S + F_\calA)$ where $(f : X \to C, S)$ is an elliptic surface with section and $(X, S + F_\calA)$ is an slc pair. An \textbf{$\calA$-weighted stable elliptic surface} is an $\calA$-weighted slc elliptic surface such that $(X, S + F_\calA)$ is a stable pair. \end{definition} 
		
		In this paper we consider only irreducible elliptic surfaces. As observed by La Nave \cite{ln}, sometimes the log minimal model program contracts the section of an slc elliptic surface.

	\begin{definition}\label{def:pseudo} A \textbf{pseudoelliptic surface} is a surface pair $(Z, F)$ obtained by contracting the section of an irreducible slc elliptic surface pair $(f:X \to C,S + \bar{F})$. By abuse of terminology, we call $F$ the marked fibers of $Z$.
		\end{definition}

\section{Relative log canonical models I: smooth generic fiber}\label{sec:local}

We begin by computing the relative log canonical models of an $\calA$-weighted elliptic surface $(f : X \to C, S + F_\calA)$ with smooth generic fiber. 

The question of whether $K_X + S + F_{\calA}$ is $f$-ample is local on the base. Therefore we may assume that $C = \Spec R$ is a DVR with closed point $s$ and generic point $\eta$. We then have the surface pair $(f: X \to C,S + aF)$ where
$
F = f^{-1}(s)_{red},
$ and $0 \leq a \leq 1$, with generic fiber $X_\eta = f^{-1}(\eta)$ a smooth elliptic curve. In particular, we can make use of the classification of central fibers $F$ in Table 1.

Since $X$ is a normal $2$-dimensional scheme and the pair is log canonical, we may replace $X$ with a minimal log resolution before running the log minimal model program over $C$. For most central fiber types, this is the unique relatively minimal smooth model. However, for the special fiber types $\mathrm{II}$, $\mathrm{III}$, $\mathrm{IV}$ the special fiber of the relatively minimal model is not normal crossings so we must further resolve to obtain a log resolution  (see Remark \ref{rmk:specialfibers}).

The section $S$ passes through the smooth locus of $f$. In particular, $S$ meets the special fiber in the smooth locus of a uniquely determined reduced fiber component. We fix notation and denote this component $A$ (denoted by a black node in the dual graph; see Table 1).

\begin{lemma}\label{canonical:In} Suppose $F$ is a fiber of type $\mathrm{I}_n$ for $n \geq 0$. Then for any $0 \leq a \leq 1$, the stable model of $(X, S + aF)$ is the surface pair obtained by contracting all components of $F$ not meeting the section so that only $A$ remains. In particular, the stable model is the Weierstrass model of $(X,S)$ with a single $A_n$ singularity. \end{lemma}

\begin{proof} Denote the components of $F$ not meeting the section by $D_i$ for $i = 1, \ldots, n - 1$. \\

\begin{center} \begin{tikzpicture}
\draw
(0,0)  circle [radius=.1] node [below] {\tiny $D_1$}
(1,0) circle [radius=.1] node [below] {\tiny $D_2$}
(3,0) circle [radius=.1] node [below] {\tiny $D_{n-2}$}
(4,0) circle [radius=.1]  node [below] {\tiny $D_{n-1}$}
(0.1,0) -- (.9,0);
\draw[dashed]
(1.1,0) -- (2.9,0);
\draw[fill=black]
(2,.5) circle [radius=.1] node [below] {\tiny $A$}
(3.1,0) -- (3.9,0)
(2.15,.45) -- (3.9,.1)
(0.1,.1) -- (1.85,.45);
\end{tikzpicture}\end{center} Note that the surface is relatively minimal so that $K_X$ is pulled back from $C$ by Proposition \ref{classcan}. This allows us to conclude that $K_X. D_i = K_X.A = 0$ for any fiber component. Furthermore, since $F$ is reduced and normal crossings, the pair $(X, S + aF)$ is log canonical. We compute 
\begin{align*} 
&(K_X + S + aF).D_i = \ 0 \\ 
&(K_X + S + aF).A = \ 1. 
\end{align*}
As a result, we must contract all components $D_i$ independent of the coefficient $a$.
\end{proof}

Next we consider a type $I_n^*$ fiber. The support of this fiber consists of $n + 5$ rational $(-2)$-curves and has dual graph affine $D_{n + 4}$, where the root corresponds to the component $A$. There is a chain of nonreduced multiplicity $2$ components $E_0, \ldots, E_n$ as well as $3$ reduced components $D_1, D_2$ and $D_3$, corresponding to the remaining vertices of valence 1. 

\begin{center}
\begin{tikzpicture}
\draw[fill=black]
(-1,-.5) circle [radius=.1] node [left] {\tiny $A$};
\draw
(-1,.5) circle [radius=.1] node [below] {\tiny $D_1$}
(0,0) circle [radius=.1] node [above] {\tiny $E_0$}
(1,0)  circle [radius=.1] node [above] {\tiny $E_i$}
(2,0)  circle [radius=.1] node [below] {\tiny $E_n$};
\draw[dashed]
(.15,0) -- (.85,0)
(1.15,0) -- (1.9,0);
\draw
(-.85,-.45) -- (-.15,-.1)
(-.85,.45) -- (-.15,.1)
(3,.5) circle [radius=.1] node [below] {\tiny $D_2$}
(3,-.5) circle [radius=.1] node [above] {\tiny $D_3$}
(2.15, .1) -- (2.85,.45)
(2.15,-.1) -- (2.85,-.45)
;
\end{tikzpicture}\end{center}

\begin{lemma}\label{canonical:In*} Suppose $F$ supports an $\mathrm{I}_n^*$ fiber. Let $\varphi : (X, S + aF) \to (Y, \varphi_*(S + aF))$ be the stable model of $(X,S + F)$ over $C$. Then we have the following:

\begin{enumerate}[label = \roman*)] 

\item If $a = 0$ then $\varphi$ contracts all components except $A$ so that $Y$ is the Weierstrass model with a single $D_{n + 4}$ singularity at the cusp of $\varphi_*A$;

\item If $0 < a < 1$ then $\varphi$ contracts all components except $E_0$ and $A$. For $n = 0$ there are three $A_1$ singularities along $\varphi_*E_0$. For $n = 1$, there is an $A_1$ singularity and an $A_3$ singularity along $\varphi_*E_0$. For $n \geq 2$, there is an $A_1$ and a $D_{n + 2}$ singularity along $\varphi_*E_0$. 

\item If $a = 1$ then $\varphi$ contracts every component except $E_0$. When $n = 0$, there are four $A_1$ singularities along $\varphi_*E_0$. When $n = 1$ there are two $A_1$ singularities and one $A_3$ singularity along $\varphi_*E_0$. When $n \geq 2$, there are two $A_1$ singularities and a $D_{n + 2}$ singularity along $\varphi_*E_0$. 

\end{enumerate}  
\end{lemma}

\begin{proof}Again since $X$ is relatively minimal over $C$, the canonical divisor $K_X \sim_{\Q,f} 0$ so that  $K_X.D = 0$ for any fiber component $D$. Furthermore, the pair $(X, S+ aF)$ is log canonical since $X$ is smooth and $S + aF$ is a normal crossings divisor. We compute
$$
(K_X + S + aF).D_i = -a
$$
for each leaf $D_i$, since $D_i$ is disjoint from all components except either $E_0$ or $E_n$. 

First suppose that $a > 0$. Then the log MMP contracts each leaf $D_i$ to obtain a log terminal model with one $A_1$ singularity along $E_0$, and two $A_1$ singularities along $E_n$. Calling this map $\mu : X \to X'$, we see  that $\varphi : X \to Y$ factors as $\varphi' \circ \mu$, where $\varphi' : X' \to Y$ is the log canonical contraction of $(X', S' + aF')$. Here, for any divisor $D$ on $X$, we denote by $D' := \mu_*D$. Now we compute  $$ (K_{X'} + S' + aF').A' = S'.A' + a(A')^2 + aA'.E_0' = 1 - a.$$ $$ (K_{X'} + S' + aF').E_i' = a(E_i')^2 + E_i'.(aE_{i - 1}' + aE_{i + 1}') = 0$$
for $i = 1, \ldots, n - 1$. We have used that $\mu$ is an isomorphism in a neighborhood of $A$ and of $E_i$ for $1 \le i \le n-1$. 

Next, we compute $\mu^*(E_0') = E_0 + 1/2D_1$ and $\mu^*(E_n') = E_n + 1/2D_2 + 1/2D_3$ so that
\begin{align*}
(E_0')^2 &= (E_0 + 1/2D_1)^2 = -3/2 \\
(E_n')^2 &= (E_n + 1/2D_2 + 1/2D_3)^2 = -1.
\end{align*}
It follows that
\begin{align*}
(K_{X'} + S' + aF').E_0' &= 1/2a \\
(K_{X'} + S' + aF').E_n' &= 0.
\end{align*}
Therefore, when $0 < a < 1$, the morphism $\varphi'$ contracts $E_i'$ for $i = 1, \ldots, n$ leaving $\varphi'_*A'$ and $\varphi'_*E_0'$. When $a = 1$, the morphism $\varphi'$ also contracts $A'$ leaving just $\varphi'_*E_0'$. 

Finally, if $a = 0$, the intersection $(K_X + S).A = 1$, and $(K_X + S).D = 0$ for any other fiber component $D$. Therefore, the morphism $\varphi$ contracts all components except $A$, and $(Y, \varphi_*(S))$ is the Weierstrass model.  The resulting singularities of $(Y, \varphi_*(S + aF))$ are deduced from the dual graphs of the trees of contracted components. 
\end{proof}

Next we consider the Kodaira fibers of type $\mathrm{II}^*, \mathrm{III}^*$ and $\mathrm{IV}^*$, which have dual graph affine $E_8, E_7$ and $E_6$ respectively.  There is a unique component $E$ of valence $3$, two leaves $D_1$ and $D_2$, and several valence $2$ components $B_j$.

\begin{prop}\label{canonical:exc*} Suppose that $F$ supports a fiber of type $\mathrm{II}^*, \mathrm{III}^*,$ or $\mathrm{IV}^*$. Let $\varphi : X \to Y$ be the stable model of $(X, S + aF)$ over $C$. Then we have the following: 

\begin{enumerate}[label = \roman*)]

\item If $a = 0$ then $\varphi$ contracts all components except $A$ so that $Y$ is the Weierstrass model,

\item if $0 < a < 1$ then $\varphi$ contracts all components except $E$ and $A$, and

\item if $a = 1$ then $\varphi$ contracts all components except $E$.  

\end{enumerate}

The singularities in each case are summarized in the table below: 

\begin{table}[h!]\label{table:sing}
\centering
\label{singularities*}
\begin{tabular}{|l|l|l|l|}
\hline
			& $a = 0$ 	& $0 < a < 1$			& $a = 1$				\\ \hline
$\mathrm{II}^*$		& $E_8$		& $A_1$, $A_2$, $A_4$	& $A_1$, $A_2$, $A_5$	\\ \hline	
$\mathrm{III}^*$		& $E_7$		& $A_1$, $A_2$, $A_3$	& $A_1$, $2A_3$			\\ \hline	
$\mathrm{IV}^*$		& $E_6$		& $A_1$, $2A_2$			& $3A_2$				\\ \hline			
\end{tabular}
\end{table}
\end{prop}

\begin{proof}[Proof of Proposition \ref{canonical:exc*}] As before, $(X, S + aF)$ is log canonical and $K_X \sim_{\Q,f} 0$ so that $K_X.D = 0$ for any fiber component $D$. We compute:
\begin{align*}
&(K_X + S + aF).A 	= 1 - a \\
&(K_X + S + aF).E	= a 	\\
&(K_X + S + aF).D_i	= -a	\\
&(K_X + S + aF).B_j = 0.
\end{align*}

If $a = 0$ then $K_X + S$ is nef and induces the morphism $\varphi$. In this case $\varphi$ contracts $E$, $D_i$ and $B_j$, and gives the Weierstrass model. 

If $a > 0$, then $K_X + S + aF$ is no longer nef. Therefore, by the log MMP there is an extremal contraction $\mu_1 : X \to X'$ contracting each leaf $D_i$. As before, for any divisor $D$ on $X$, we denote by $D' := \mu_{1*}D$. Suppose one such $D_i$ meets the valence $3$ component $E$. Then 
$$
\mu_1^*E' = E + 1/2D_i
$$
so that by the projection formula
$$
(E')^2 = (E + 1/2D_i).E = -3/2.
$$
This allows us to compute
$$
(K_{X'} + S' + aF').E' = 1/2a.
$$
If $E$ does not meet any $D_i$, then since $\mu_1$ is an isomorphism in a neighborhood of $E$:
$$
(K_{X'} + S' + aF').E' = a.
$$
Let $B$ be a valence $2$ component of $F$ on $X$ that meets one of $D_i$. Then $B'$ is a valence one component passing through an $A_1$ singularity, and therefore
$$
\mu_1^*(B') = B + 1/2D_i
$$
for some $i$, from which it follows that
$$
(B')^2 = (B + 1/2D_i)^2 = -3/2.
$$
Now we can compute
$$
(K_{X'} + S' + aF').B' = -1/2a.
$$
On the other hand, if $B$ is a valence $2$ component of $F$ not meeting either $D_1$ or $D_2$, then $\mu$ is an isomorphism in a neighborhood of $B$ so that
$$
(K_{X'} + S' + aF').B' = 0.
$$
Thus we must perform another extremal contraction $\mu_2 : X' \to X''$ that contracts exactly the valence $2$ components that meet either $D_1$ or $D_2$, producing $A_2$ singularities. 

For type $II^*$ fiber there is exactly one of these, denoted $B_1$, meeting $D_1$, and then there are 4 valence $2$ components ($B_2, \dots, B_5$) not meeting $D_1$ or $D_2$ (see figure below).

\begin{center}
\begin{tikzpicture}
\draw
(0,0) circle [radius=.1] node [below] {\tiny $D_1$}
(.5,0)  circle [radius=.1] node [below] {\tiny $B_1$}
(1,0)  circle [radius=.1] node [below] {\tiny $E$}
(1,.5) circle [radius=.1] node [right] {\tiny $D_2$}
(1.5,0)  circle [radius=.1] node [below] {\tiny $B_2$}
(2,0)  circle [radius=.1] node [below] {\tiny $B_3$}
(2.5,0)  circle [radius=.1] node [below] {\tiny $B_4$}
(3,0)  circle [radius=.1] node [below] {\tiny $B_5$}
(1,.1) -- (1,.4)
(.1,0) -- (.4,0)
(.6,0) -- (.9,0)
(1.1,0) -- (1.4,0)
(1.6,0) -- (1.9,0)
(2.1,0) -- (2.4,0)
(2.6,0) -- (2.9,0)
(3.1,0) -- (3.4,0);
\draw[fill=black]
(3.5,0) circle [radius=.1] node [below] {\tiny $A$};
\end{tikzpicture}
\end{center}

 As before we use the notation $\mu_{2*}D' = D''$, and then 
$$
(\mu_2 \circ \mu_1)^*E'' = E + 1/3D_1 + 1/2D_2 + 2/3B_1
$$
so that by the projection formula,
$$
(E'')^2 = (E + 1/3D_1 + 1/2D_2 + 2/3B_1).E = -5/6.
$$
The $\mu_i$ are isomorphisms in a neighborhood of $A, B_2, \ldots, B_5$. This lets us compute:
\begin{align*}
&(K_{X''} + S'' + aF'').A'' = 1 - a \\
&(K_{X''} + S'' + aF'').B_i'' = 0	\\
&(K_{X''} + S'' + aF'').E'' = 1/6a	
\end{align*}
so that $(K_{X''} + S'' + aF'')$ is $f$-nef and thus $f$-semiample by abundance (Proposition \ref{prop:abundance}).

We take the log canonical contraction $\varphi'': X'' \to Y$ over $C$ to obtain the stable model with $\varphi = \varphi'' \circ \mu_2 \circ \mu_1$. If $0 < a < 1$ then $\varphi''$ contracts $B_i''$ and if $a = 1$ $\varphi''$ contracts $B_i''$ and $A''$ completing the claim. 

In a type $\mathrm{III}^*$ fiber there is exactly one valence $2$ component $B_1$ meeting $D_1$, and valence $2$ components $B_2, B_3, B_4$ meeting neither $D_1$ nor $D_2$ as below:

\begin{center}
\begin{tikzpicture}
\draw
(0,0) circle [radius=.1] node [below] {\tiny $D_1$}
(.5,0)  circle [radius=.1] node [below] {\tiny $B_1$}
(1,0)  circle [radius=.1] node [below] {\tiny $B_2$}
(1.5,0)  circle [radius=.1] node [below] {\tiny $E$}
(1.5,.5) circle [radius=.1] node [right] {\tiny $D_2$}
(2,0)  circle [radius=.1] node [below] {\tiny $B_3$}
(2.5,0)  circle [radius=.1] node [below] {\tiny $B_4$}
(1.5,.1) -- (1.5,.4)
(.1,0) -- (.4,0)
(.6,0) -- (.9,0)
(1.1,0) -- (1.4,0)
(1.6,0) -- (1.9,0)
(2.1,0) -- (2.4,0)
(2.6,0) -- (2.9,0);
\draw[fill=black]
(3,0) circle [radius=.1] node [below] {\tiny $A$};
\end{tikzpicture}\end{center}

In this case we have
$$
(\mu_2 \circ \mu_1)^*B_2'' = B_2 + 2/3B_1 + 1/3D_1,
$$
and $\mu_2$ is an isomorphism away from $B_2''$. By the projection formula,
$$
(B_2'')^2 = (B_2 + 2/3B_1 + 1/3D_1).B_2 = -4/3.
$$
Then we can compute:
\begin{align*}
&(K_{X''} + S'' + aF'').A'' = 1 - a 	\\
&(K_{X''} + S'' + aF'').B_2'' = -1/3a	\\
&(K_{X''} + S'' + aF'').B_3'' = 0		\\
&(K_{X''} + S'' + aF'').B_4'' = 0		\\
&(K_{X''} + S'' + aF'').E'' = 1/2a		
\end{align*}
Therefore we must perform a third extremal contraction $\mu_3: X'' \to X'''$, that contracts $B_2''$ and results in an $A_3$ singularity along $E'''$.  We have
$$
(\mu_3 \circ \mu_2 \circ \mu_1)^*E''' = E + 1/2D_2 + 1/4D_1 + 1/2B_1 + 3/4B_2
$$
and so by the projection formula:
$$
(E''')^2 = (E + 1/2D_2 + 1/4D_1 + 1/2B_1 + 3/4B_2).E = -3/4
$$
 Therefore
\begin{align*}
&(K_{X'''} + S''' + aF''').A''' = 1 - a 	\\
&(K_{X'''} + S''' + aF''').B_3''' = 0		\\
&(K_{X'''} + S''' + aF''').B_4''' = 0		\\
&(K_{X'''} + S''' + aF''').E''' = 1/4a.		
\end{align*}

It follows that $K_{X'''} + S''' + aF'''$ is nef and thus $f$-semiample by abundance (Proposition \ref{prop:abundance}), so that there is a log canonical contraction $\varphi''' : X''' \to Y$ to the stable model over $C$ so that $\varphi = \mu_3 \circ \mu_2 \circ \mu_1 \circ \varphi'''$. If $0 < a < 1$, then $\varphi'''$ contracts $B_3'''$ and $B_4'''$ and if $a = 1$, then $\varphi'''$ contracts $B_i'''$ and $A'''$ as claimed. 

Finally, in a type $\mathrm{IV}^*$ fiber there are two valence $2$ components $B_1$ and $B_2$ meeting $D_1$ and $D_2$ respectively, and a valence $2$ component $B_3$ meeting neither component:  

\begin{center}
\begin{tikzpicture}
\draw
(0,0) circle [radius=.1] node [below] {\tiny $D_1$}
(.5,0)  circle [radius=.1] node [below] {\tiny $B_1$}
(1,0)  circle [radius=.1] node [below] {\tiny $E_1$}
(1.5,0)  circle [radius=.1] node [below] {\tiny $B_3$}
(1.0,.5) circle [radius=.1] node [right] {\tiny $B_2$}
(1.0,1) circle [radius=.1]  node [right] {\tiny $D_2$}
(1.0,.1) -- (1.0,.4)
(1.0,.6) -- (1.0, .9)
(.1,0) -- (.4,0)
(.6,0) -- (.9,0)
(1.1,0) -- (1.4,0)
(1.6,0) -- (1.9,0);
\draw[fill=black]
(2,0)  circle [radius=.1] node [below] {\tiny $A$};
\end{tikzpicture}
\end{center}

We have two $A_2$ singularities along $E''$. Therefore
$$
(\mu_2 \circ \mu_1)^*E'' = E + 1/3D_1 + 1/3D_2 + 2/3B_1 + 2/3B_2
$$
and by the projection formula,
$$
(E'')^2 = (E + 1/3D_1 + 1/3D_2 + 2/3B_1 + 2/3B_2).E = -2/3.
$$
Furthermore, $\mu_i$ are isomorphisms in a neighborhood of $B_3$ and $A$. This lets us compute:
\begin{align*}
&(K_{X''} + S'' + aF'').A'' = 1 - a \\
&(K_{X''} + S'' + aF'').B_3'' = 0	\\
&(K_{X''} + S'' + aF'').E'' = 1/3a.	
\end{align*}

Thus $K_{X''} + S'' + aF''$ is $f$-nef and thus $f$-semiample by Proposition \ref{prop:abundance}, so there is a log canonical contraction $\varphi'' : X'' \to Y$ to the stable model over $C$ so that $\varphi = \mu_2 \circ \mu_1 \circ \varphi''$. If $0 < a < 1$, then $\varphi''$ is the contraction of $B_3''$ and if $a = 1$, then $\varphi''$ contracts both $B_3''$ and $A''$.\end{proof}

Finally, we are left with type $II, III$, $\mathrm{IV}$ fibers in Kodaira's classification. 

\begin{remark} \label{rmk:specialfibers} These fibers $F$ are such that $S + F$ is \emph{not} a normal crossings divisor. As such, $(X, S + aF)$ does not have log canonical singularities for $a > 0$, and so we must first take a log resolution $q : Z \to X$ before we can run the log MMP with the pair $(Z, S + a\widetilde{F} + \mathrm{Exc}(q))$ to obtain the stable model. Here $\widetilde{F}$ is the strict transform of $F$, and we note that $(Z, S + a\widetilde{F} + \mathrm{Exc}(q))$ is log canonical.\end{remark}

The dual graph of the special fiber of the log resolution looks as follows in each case:

\begin{center}
\begin{tikzpicture}
\draw
(.5,0)  circle [radius=.1] node [below] {\tiny $D_1$}
(1,0)  circle [radius=.1] node [below] {\tiny $E$};
\draw[fill=black]
(1.5,0)  circle [radius=.1] node [below] {\tiny $A$};
\draw
(1.0,.5) circle [radius=.1] node [right] {\tiny $D_2$}
(1.0,.1) -- (1.0,.4)
(.6,0) -- (.9,0)
(1.1,0) -- (1.4,0);
\end{tikzpicture}
\end{center}

However, the multiplicities and self intersections of the components vary depending on the type. Furthermore, since $Z$ is \emph{not} minimal over $C$, then $K_Z \not \sim_{\Q,f} 0$. Rather, the canonical class depends on the number of blowups performed to obtain the log resolution.

\begin{prop}\label{canonical:II} Suppose $F$ supports a fiber of type $\mathrm{II}$ and let $q : Z \to X$ as above. Let $\varphi : Z \to Y$ be the stable model of $(Z, S + a\widetilde{F} + \mathrm{Exc}(q))$ over $C$. Then: 
\begin{enumerate}[label = \roman*)]
\item if $0 \leq a \le 5/6$ then $\varphi$ contracts all components except $A$ so that $Y$ is the Weierstrass model of $(X,S)$,
\item if $5/6 < a < 1$ then $\varphi$ contracts $D_1$ and $D_2$, and
\item if $a = 1$ then $\varphi$ contracts all components except $E$. 
\end{enumerate} 
\end{prop}

\begin{proof} The minimal log resolution $q$ is obtained by three successive blowups so that $D_1, D_2$ and $E$ are exceptional divisors with self intersection $-3$, $-2$, and $-1$ respectively. Furthermore, one can compute that $A$ is a $-6$ curve. Here $\widetilde{F} = A$ and $\mathrm{Exc}(q) = D_1 + D_2 + E$. The canonical bundle is given by $K_Z = q^*(K_X) + D_1 + 2D_2 + 4E$. We compute:
\begin{align*}
&(K_Z + S + aA + D_1 + D_2 + E).A = 6 -6a	\\
&(K_Z + S + aA + D_1 + D_2 + E).D_1 = -1	\\
&(K_Z + S + aA + D_1 + D_2 + E).D_2 = -1	\\
&(K_Z + S + aA + D_1 + D_2 + E).E = a
\end{align*}

Therefore, there is an extremal contraction $\mu : Z \to Z'$ contracting $D_1$ and $D_2$. Denote by $\Delta' := \mu_{*}\Delta$ for any divisor $\Delta$ on $Z$. Then $\mu^* E = E + 1/3D_1 + 1/2D_2$ and $\mu^*(K_{Z'}) = K_Z + 1/3D_1$. Using the projection formula, we calculate
\begin{align*}
&(E')^2 = (E + 1/3D_1 + 1/2D_2).E = -1/6	\\  
&K_{Z'}.E' = (K_Z + 1/3D_1).E = -2/3,
\end{align*}
which gives
\begin{align*}
&(K_{Z'} + S' + aA' + E').E' = a - 5/6		\\
&(K_{Z'} + S' + aA' + E').A' = 6 - 6a.
\end{align*}

When $0 \le a < 5/6$,  there is another extremal contraction $\varphi : Z' \to Y$ contracting $E'$ resulting in a stable model $Y$ that is the Weierstrass model. When $a = 5/6$, we see that $K_{Z'} + S' + aF'$ is $f$-nef  and thus $f$-semiample by abundance (Proposition \ref{prop:abundance}). The log canonical contraction contracts $E'$ resulting in the Weierstrass model $\varphi : Z' \to Y$. In either case the stable model is the same, but $(Y, \varphi_*(S' + aA'))$ has log terminal (resp. log canonical) singularities for $a < 5/6$ (resp $a = 5/6$). 

When $5/6 < a < 1$, then $K_{Z'}  + S' + aF' + E'$ is ample so that $(Z', S' + aF' + E')$ is the stable model over $C$. This leaves the case $a = 1$: here $K_{Z'} + S' + aF' + E'$ is $f$-nef and thus $f$-semiample by (abundance) Proposition \ref{prop:abundance}, and the log canonical contraction $\varphi: Z' \to Y$ contracts $A'$ leaving just $\varphi_*E$. \end{proof}

\begin{prop}\label{canonical:III} Suppose $F$ supports a fiber of type $\mathrm{III}$ and let $q : Z \to X$ as above. Let $\varphi: Z \to Y$ be the stable model of $(Z, S + a\widetilde{F} + \mathrm{Exc}(q))$ over $C$. Then:.
\begin{enumerate}[label = \roman*)]
\item If $0 \leq a \leq 3/4$ then $\varphi$ contracts all components except for $A$ resulting in the Weierstrass model, 
\item if $3/4 < a < 1$ then $\varphi$ contracts $D_1$ and $D_2$, and 
\item if $a = 1$ then $\varphi$ contracts all components except $E$.
\end{enumerate}
\end{prop}

\begin{proof}The minimal log resolution $q : Z \to X$ is obtained by two successive blowups at the point of tangency. We have exceptional divisors $\mathrm{Exc}(q) = D_2 + E$ and $\widetilde{F} = A + D_1$ with self intersections $A^2 = D_1^2 = -4$, $D_2^2 = -2$ and $E^2 = -1$. Furthermore $K_Z = q^*K_X + D_2 + 2E$. Then:
\begin{align*}
&(K_Z + S + a\widetilde{F} + \mathrm{Exc}(q)).A = 4 -4a		\\
&(K_Z + S + a\widetilde{F} + \mathrm{Exc}(q)).D_1 = 3 - 4a	\\
&(K_Z + S + a\widetilde{F} + \mathrm{Exc}(q)).D_2 = -1		\\
&(K_Z + S + a\widetilde{F} + \mathrm{Exc}(q)).E = 2a -1 	
\end{align*}

Suppose that $0 \leq a < 1/2$. In this case there is an extremal contraction blowing back down $E$ and $D_2$-- this is precisely the blow down $q : Z \to X$. Denoting by $\Delta' := q_*\Delta$ for any divisor $\Delta$ on $Z$, we have \begin{align*}
&(K_{X} + S' + a(A' + D_1')).A' = 1 	\\
&(K_{X} + S' + a(A' + D_1')).D_1' = 0	
\end{align*}
so that $K_X + S' + a(A' + D_1')$ is $f$-nef and thus $f$-semiample by abundance (Proposition \ref{prop:abundance}). The log canonical contraction $\varphi' : X \to Y$ contracts $D_1'$ resulting in the Weierstrass model. 

Now let $1/2 \le a < 3/4$, in which case the first extremal contraction $\mu_1: Z \to Z'$ contracts $D_2$. Note that this is an isomorphism away from $E$, and we calculate that $\mu_1^*(E') = E + 1/2 D_2$, giving $(E')^2 = -1/2$. Since $\mu_1^*K_{Z'} = K_Z$, we also have that $K_{Z'}.E' = K_Z. E = -1$.  Thus: 
$$
(K_{Z'} + S' + a(A' + D_1') + E').E' = 2a - 3/2.
$$
Since $ a < 3/4$, there is a second extremal contraction $\mu_2 : Z' \to X$ contracting $E'$ and resulting in the minimal model $X$ again. As before, the $K_X + S'' + a(A'' + D_1'')$ is $f$-semiample and the log canonical contraction blows down $D_1''$ to obtain the Weierstrass model.

When $a = 3/4$, we perform the first extremal contraction of $D_2$ via $\mu_1: Z \to Z'$ as above. Then $K_{Z'} + S' + a(A' + D_1') + E'$ is $f$-nef and the log canonical contraction will contract $E'$ and $D_1'$ resulting in the Weierstrass model. 

Now suppose $3/4 < a < 1$; the first extremal contraction $\mu:Z \to Z'$ contracts $D_1$ and $D_2$ (since $3-4a < 0$). We compute  $\mu^*(E') = E + 1/2D_2 + 1/4D_1$ and $\mu^*K_Z' = K_Z + 1/2D_1$, so 
\begin{align*}
&E'^2 = (E + 1/2D_2 + 1/4D_1).E = -1 + 1/2 + 1/4 = -1/4	\\
&K_Z'.E' = (K_Z + 1/2B).E = -1/2.
\end{align*}
This allows us to recompute 
$$
(K_{Z'} + S' + aA' + E'). E' = a - 3/4 >0.
$$ 
 Therefore $K_{Z'} + S' + aA' + E'$ is $f$-ample and $Z'$ is the stable model over $C$. 

Finally, suppose that $a = 1$. As above, there is an extremal contraction $\mu : Z \to Z'$ contracting $D_1$ and $D_2$ but now $(K_{Z'} + S' + A' + E').A = 0$ so that the log canonical contraction $\varphi: Z' \to Y$ contracts $A$. 
\end{proof}

\begin{prop}\label{canonical:IV} Suppose $F$ supports a fiber of type $\mathrm{IV}$ and let $q : Z \to X$ be the minimal log resolution. Let $\varphi : Z \to Y$ be the stable model of $(Z, S + a\widetilde{F} + \mathrm{Exc}(q))$ over $C$. Then: 
\begin{enumerate}[label = \roman*)]
\item if $0 \leq a \leq 2/3$ then $\varphi$ contracts all components except $A$ resulting in the Weierstrass model,
\item if $2/3 < a < 1$ then $\varphi$ contracts the leaves $D_i$, and
\item if $a = 1$, $\varphi$ contracts all components except $E$. 
\end{enumerate}
\end{prop}

\begin{proof} The minimal log resolution $q : Z \to X$ is the blowup of the triple point. The exceptional divisor $\mathrm{Exc}(q) = E$ and $\widetilde{F} = A + D_1 + D_2$. The fiber components in $X$ are $-2$ curves therefore their strict transforms $A$ and $D_i$ are $-3$ curves. Furthermore $K_Z = q^*(K_X) + E$. Therefore $K_Z.A = K_Z.D_i = 1$ and $K_Z.E = -1$. It follows that
\begin{align*}
&(K_Z + S + aF + E).A = 3 -3a\\
&(K_Z + S + aF + E).D_i = 2 - 3a\\
&(K_Z + S + aF + E).E = 3a - 2.
\end{align*}

We begin with the case $0 \leq a < 2/3$. Here, we see that $K_Z  + S + aF + E$ is not nef and there exists an extremal contraction of $E$. This is precisely the blowup $q : Z \to X$. Denoting $\Delta' = q_*\Delta$ for any divisor $\Delta$ on $Z$, we have $(A')^2 = (D'_i)^2= -2.$ Therefore
\begin{align*}
&(K_{Z'} + S' + aF').A' = 1\\
&(K_{Z'} + S' + aF').D'_i = 0.
\end{align*}
Here we used that $K_X.\Delta = 0$ for any fiber component $\Delta$ as $X$ is relatively minimal over $C$. 

Now, we see that $K_{X} + S' + aF'$ is $f$-semiample and the morphism $\varphi' : X \to Y$ contracts $D'_i$ resulting in the Weierstrass model. 

When $a = 2/3$ we have a similar outcome. In this case $K_Z + S + aF + E$ is  already and the morphism $\varphi : Z \to Y$ contracts $D_i$ and $E$ so again we obtain the Weierstrass model. The difference is that for $a = 2/3$, the singularities of the stable model are strictly log canonical while they are log terminal for $a < 2/3$. 

Next suppose that $ 2/3 < a < 1$. In this case there is an extremal contraction $\mu Z \to Z'$ that contracts the $D_i$. Note that $\mu^*(E') = E + 1/3D_1 + 1/3D_2$, thus by the projection formula $(E')^2 = -1/3$. Furthermore, $\mu^*K_{Z'} = K_Z + 1/3D_1 + 1/3D_2$. Since $K_Z \sim_{\Q,f} E$ then by pushing forward we see that $K_{Z'} \sim_{\Q,f'} E'$ where $f' Z' \to C$ is the elliptic fibration; that is $E'$ and $K_{Z'}$ are rationally equivalent over $C$. It follows that $K_{Z'}.A' = 1$ and $K_{Z'}.E = -1/3$.

This allows us to compute:
\begin{align*}
&(K_{Z'} + S' + aF' + E').A' = 3 - 3a\\
&(K_{Z'} + S' + aF' + E').E' = a - 2/3 
\end{align*}
Therefore, $K_{Z'} + S' + aF' + E'$ is $f'$-ample $\mu : Z \to Z'$ is the stable model over $C$. \\

Finally, suppose that $a = 1$. As above, there is an extremal contraction $\mu : Z \to Z'$ contracting the $D_i$. Based on the above calculation, $K_{Z'} + S' + aF' + E'$ is $f'$-nef and $f'$-semiample and the morphism $\varphi' : Z' \to Y$ contracts $A'$. This leaves just $\varphi'_*E'$. 
\end{proof}

In each of fibers type $II, III$ and $\mathrm{IV}$, the stable model over $C$ is the Weierstrass model for $a \le a_0$ where $a_0 = 5/6, 3/4, 2/3$ respectively. We summarize the singularities obtained in the log canonical model for these fibers: 

\begin{table}[h!]
\centering
\label{singularities}
\begin{tabular}{|l|l|l|l|}
\hline
			& $0 \le a \le a_0$		& $a_0 < a < 1$		& $a = 1$					\\ \hline
$\mathrm{II}$		& $A_0$					& $A^*_1$, $A^*_2$	& $A^*_1$, $A^*_2$, $A^*_5$	\\ \hline	
$\mathrm{III}$		& $A_1$					& $A^*_1$, $A^*_3$	& $A^*_1$, $2A^*_3$			\\ \hline	
$\mathrm{IV}$		& $A_2$					& $2A^*_2$			& $3A^*_2$					\\ \hline	

\end{tabular}
\end{table}
Here $A_0$ denotes a smooth point at the cusp of the central fiber and $A_{n-1}^*$ denotes the singularity obtained by contracting a rational $-n$ curve on a smooth surface. We use this dual notation suggestively -- this is further discussed in \cite{tsm}. 

\begin{remark} The numbers $a_0$ above can easily be seen to be the log canonical thresholds of the Weierstrass model with respect to the the corresponding central fiber. \end{remark} 

\begin{definition}\label{def:twisted1} Given a relative log canonical model of an elliptic surface with section $f : X \to C$, we say that a fiber of $f$ is a \textbf{twisted fiber} if it is irreducible but non-reduced. 
A fiber is called an \textbf{intermediate fiber} if it is a nodal union of a reduced component $A$ and a non-reduced component $E$ such that the section meets the fiber along the smooth locus of $A$. \end{definition}

\begin{remark} By the computations of this section, we see that the following  are equivalent:
\begin{itemize}
\item log canonical models at $a = 1$ of fibers that are not of type $\mathrm{I}_n$;
\item twisted fibers.
\end{itemize}
We can summarize the results above as stating that as we vary the coefficient $a$ of the central fiber, the log canonical model interpolates between a twisted fiber at $a = 1$ and the Weierstrass fiber at $a = 0$. Below, we will see the same behavior in the non-normal case.  \end{remark}

\section{Relative log canonical models II: nodal generic fiber}\label{sec:jinfty}
We now turn to elliptic surfaces with nodal generic fiber. These must necessarily have constant $j$-invariant $\infty$. Such an elliptic surface is non-normal with normalization a birationally ruled surface over the same base curve. 

As above we let $X \to \Spec R = C$ be a flat elliptic fibration with section $S$ over the spectrum of a complete DVR. As usual $F$ denotes the reduced divisor corresponding to the central fiber. We begin with the Weierstrass fibrations (see also \cite[Lemma 3.2.2]{ln}): 

\begin{lemma} A Weierstrass fibration $f : X \to \Spec R$ with nodal generic fiber has equations $y^2 = x^2(x - \lambda t^k)$ where $t$ is the uniformizer in $R$ and $\lambda$ is a unit. Furthermore, $(X,S)$ is an slc pair if and only if $k \le 2$. \end{lemma} 

\begin{proof}The form of the Weierstrass equation is given in \cite[Lemma 3.2.2]{ln}.  The section $S$ is a smooth divisor passing through the smooth locus of $f$ so $(X,S)$ is slc if and only if $(X,0)$ is slc. Let $\nu : X^\nu \to X$ be the normalization. Then $X^\nu$ is defined by $w^2 = x - \lambda t^k$ and $\nu$ is induced by the homomorphism
\begin{align*}
\frac{k[x,y,t]}{(y^2 - x^2(x - \lambda t^k))} &\rightarrow \frac{k[w,x,t]}{(w^2 - (x - \lambda t^k))} \\ 
y &\mapsto wx 
\end{align*}
One can check that $X^\nu$ is a smooth surface. The double locus $x = y = 0$ in $X$ pulls back to the locus $x = 0$ in $X^\nu$. This is the divisor $D^\nu = \{w^2 = -\lambda t^k\}$. The pair $(X^\nu, D^\nu)$ is slc if and only if $D^\nu$ is an at worst nodal curve and this occurs exactly when $k \le 2$.  
\end{proof}

After reparametrizing, we may suppose $\lambda = 1$. Next we compute the fiber types that appear in minimal log semi-resolutions of these slc Weierstrass models analogous to Kodaira's classification: 

\begin{lemma}\label{slckodaira}
Consider the equation $y^2 = x^2(x-t^k)$ as above.
 \begin{enumerate}[label = \roman*)]

\item $k = 0$: $y^2 = x^2(x - 1)$ is a semi-smooth surface and the elliptic fibration is a trivial family with all fibers nodal cubics;

\item $k = 1$: the minimal log semi-resolution of $y^2 = x^2(x - t)$ is an elliptic surface $f : Y \to \Spec R$ where the reduced central fiber is a nodal chain of rational curves $A,B,E$ and $E$ supports a multiplicity two fiber component intersecting both $A$ and $B$;

\item $k = 2$: the minimal log semi-resolution of $y^2 = x^2(x - t^2)$ is an elliptic surface with $f : Y \to \Spec R$ with central fiber a nodal union of $E$ and $A$ where $A$ is a reduced rational curve and $E$ is a nodal cubic. 

\end{enumerate} 
\end{lemma} 

\begin{proof}

$(i)$ is clear. 

For $(ii)$, we blow up once at $(0,0,0)$ to obtain a surface with two central fiber components: a nonreduced component of multiplicity 2 supported on the exceptional divisor of the blowup, and a reduced rational component given by the strict transform of the central fiber of the Weierstrass model. In local coordinates, two of the charts are smooth and the relevant chart is $u^2 = v^2t(u - 1)$ which has an $A_1$ singularity at $(u,v,t) = (1,0,0)$ but is semi-smooth elsewhere. Blowing this up yields the semi-resolution as described. 

For $(iii)$ we take the normalization of $y^2 = x^2(x - t^2)$ as in the proof of the above lemma. This is the smooth surface $w^2 = x - t^2$ with $wx = y$ and double locus $w^2 = - t^2$ the union of two components $D_i$ intersecting at $(x,w,t) = (0,0,0)$. The central fiber of the normalization is the rational curve $w^2 = x$. Then blow up $(0,0,0)$ to obtain a rational surface $X'$ and let $E'$ be the exceptional divisor, $A'$ the strict transform of the fiber and $D_i'$ the strict transforms of the double locus. We may glue back together the $D_i'$ to obtain a map $\mu : X' \to X$. Then $X$ is a semi-smooth elliptic surface resolving our Weierstrass fibration and the central fiber consists of $E = \mu_*(E')$ and $A = \mu_*(A')$ as described. \end{proof} 
\begin{definition} \label{inftyfibers} The fibers $N_k$ are the slc fiber types with Weierstrass equation $y^2 = x^2(x - t^k)$ for $k = 0,1,2$. \end{definition}

The $N_0$ case is clear:

\begin{lemma} Let $f : X \to C$ with central fiber $F$ of type $\mathrm{N}_0$. Then $(X, S + aF)$ is relatively stable over $C$ for all $0 \le a \le 1$. \end{lemma} 

\begin{prop} Suppose $f : X \to C$ is a type $\mathrm{N}_1$ fiber. Let $q : Z \to X$ be the minimal log resolution. Let $\varphi : Z \to Y$ be the stable model of $(Z, S + a \widetilde{F} + \mathrm{Exc}(q))$ over $C$. Then: 
\begin{enumerate}[label = \roman*)]
\item when $0 \leq a \leq 1/2$, $\varphi$ contracts all components except $A$ giving the Weierstrass model;
\item when $1/2 < a < 1$, then $\varphi$ contracts all components except for $A$ and $E$; and
\item if $a = 1$, then $\varphi$ contracts all components except $E$. 
\end{enumerate}
\end{prop}

\begin{proof} In the notation of Lemma \ref{slckodaira}, $\widetilde{F} = A$ and $\mathrm{Exc}(q) = B + E$. The central fiber is $A + B + 2E$. One can check that $A^2 = B^2 = -2$ and $E^2 = -1$ and that $K_Z.A = K_Z.B = 0$, $K_Z.E = 0$ by adunction (\cite[Proposition 4.6]{ksb}). We compute:
\begin{align*}
&(K_Z + S + aA + B + E).A = 2 - 2a \\
&(K_Z + S + aA + B + E).B = -1 \\
&(K_Z + S + aA + B + E).E = a. 
\end{align*}

There is an extremal contraction $\mu : Z \to Z'$ contracting the $(-2)$ curve $B$. Letting $D' = \mu_*D$ for any divisor $D$ on $Z$, then $(E')^2 = -1/2$ and the other intersection numbers remain unchanged since $\mu$ is crepant. Thus
\begin{align*}
&(K_{Z'} + S' + aA' + E').A' = 2 - 2a \\
&(K_{Z'} + S' + aA' + E').E' = a - 1/2. 
\end{align*}

Therefore when $a = 1$, there is a semi-log canonical contraction $\varphi : Z' \to Y$ contracting $A'$. When $1/2 < a < 1$ then $Z'$ is the stable model over $C$ and $\varphi = \mu$. When $a = 1/2$ there is a semi-log canonical contraction $\mu : Z' \to Y$ that contracts $E'$ obtaining the Weierstrass model. Finally when $a < 1/2$, there is still an extremal contraction $\mu : Z' \to Y$ contracting $E'$ yielding again the Weierstrass model. \end{proof}

\begin{prop}\label{canonical:k2} Suppose $f : X \to C$ is a type $\mathrm{N}_2$ Weierstrass model, and let $q : Z \to X$ be the minimal semi-log resolution as in Lemma \ref{slckodaira} with reduced central fiber $F$. Then the stable model $\varphi : Z \to Y$ of $(Z, S + aF)$ over $C$ is a type $N_0$ Weierstrass model. 
\end{prop}

\begin{proof} The central fiber of $Z \to C$ is reduced with $F = A + E$ where $A$ is rational and $E$ is a nodal cubic such that $A.E = 1$.  By adjunction (\cite[Proposition 4.6]{ksb}) we deduce that $K_Z.E = 1$ and $K_Z.A = -1$. Furthermore, we must have $A^2 = E^2 = -1$ so we see that
\begin{align*}
&(K_Z + S + a(A + E)).A = -1 + 1 -a + a = 0 \\
&(K_Z + S + a(A + E)).E = 1.
\end{align*}
Therefore, $K_Z + S + aF$ is relatively semiample over $C$ and the semi-log canonical contraction $\varphi : Z \to Y$ contracts $A$ yielding a Weierstrass model of type $\mathrm{N}_0$. \end{proof}

\section{Canonical Bundle Formula}\label{sec:canonicalbundle} Using the above results, we can now compute the canonical bundle of a relatively stable elliptic surface pair. 
 
\begin{theorem}\label{canonical} Let $f: X \to C$ be a fibration where $X$ is an irreducible elliptic surface with section $S$. Furthermore, let $F_{\calA} = \sum a_i F_i$ be a sum of reduced marked fibers $F_i$ with $0 \le a_i \le 1$. Suppose that $(X, S + F_{\calA})$ is the relative semi-log canonical model over $C$. Then 
$$
\omega_X = f^*(\omega_C \otimes \LL) \otimes \calO_X(\Delta).
$$
where $\Delta$ is effective and supported on fibers of type $\mathrm{II}, \mathrm{III}$, and $\mathrm{IV}$ contained in $\Supp(F)$. The contribution of a type $\mathrm{II}$, $\mathrm{III}$ or $\mathrm{IV}$ fiber to $\Delta$ is given by $\alpha E$ where $E$ supports the unique nonreduced component of the fiber and
$$
\alpha = \left\{ \begin{array}{lr} 4 & \mathrm{II} \\ 2 & \mathrm{III} \\ 1 & \mathrm{IV} \end{array} \right. 
$$
\end{theorem}

It is important to emphasize here that only type $\mathrm{II}$, $\mathrm{III}$ or $\mathrm{IV}$ fibers that are \emph{not} in Weierstrass form affect the canonical bundle. If all of the type $\mathrm{II}$, $\mathrm{III}$, and $\mathrm{IV}$ fibers of $f : X \to C$ are Weierstrass, then the usual canonical bundle formula $\omega_X = f^*(\omega_C \otimes \LL)$ holds. \\

Before proceeding with the proof, we will need the following two lemmas:

\begin{lemma}\label{pushpull} Let $X$ be seminormal and $\mu : Y \to X$ a projective morphism with connected fibers. Then for any coherent sheaf $\mathcal{F}$ on $X$, we have that $\mu_*\mu^* \mathcal{F} = \mathcal{F}$. \end{lemma} 

\begin{proof} Note that $\mu_* \calO_Y = \calO_X$ by the defining property of being seminormal. The result then follows by the projection formula. \end{proof} 

\begin{lemma} \label{logcanonical} Let $(X, \Delta)$ be an slc pair and $\mu :Y \to X$ a partial semi-resolution. Write
$$
K_Y + \mu_*^{-1}\Delta + \Gamma = \mu^*(K_X + \Delta) + B
$$
where $\Gamma = \sum_i E_i$ is the exceptional divisor of $\mu$ and $B$ is effective and exceptional. Then
$$
\mu_*\calO_Y(m(K_Y + \mu_*^{-1}\Delta + \Gamma)) \cong \calO_X(m(K_X + \Delta)). 
$$
\end{lemma}

\begin{proof} 
There is an exact sequence 
$$
0 \to \mu^*\calO_X(m(K_X + \Delta)) \to \calO_Y(m(K_Y + \mu_*^{-1}\Delta + \Gamma)) \to \calO_{mB}(mB|_{mB}) \to 0.
$$
If $B = 0$ then $\mu^*\calO_X(m(K_X + \Delta)) \cong \calO_Y(m(K_Y + \mu_*^{-1}\Delta + \Gamma))$. Otherwise $B^2 < 0$, since $B \geq 0$ is exceptional and the intersection form on exceptional curves is negative definite \cite[Theorem 10.1]{singmmp}. Therefore $\calO_{mB}(mB|_{mB})$ has no sections and so $\mu_*\calO_{mB}(mB|_{mB}) = 0$. In either case, 
$$
\mu_*\mu^*\calO_X(m(K_X + \Delta)) \cong \mu_*\calO_Y(m(K_Y + \mu_*^{-1}\Delta + \Gamma)).
$$
On the other hand, $\mu_*\mu^*\calO_X(m(K_X + \Delta)) = \calO_X(m(K_X + \Delta))$ by Lemma \ref{pushpull}.  \end{proof}  

\begin{proof}[Proof of Theorem \ref{canonical}] The formula is true for Weierstrass fibrations by \cite[Proposition III.1.1]{mir3}. These include fiber types $\mathrm{N}_0, \mathrm{N}_2$ and $\mathrm{I}_n$ for any coefficients as well as the relative log canonical models of any fiber with coefficient $a = 0$. For marked fibers of type $\mathrm{I}_n, \mathrm{I}_n^*, \mathrm{II}^*, \mathrm{III}^*$, and $\mathrm{IV}^*$ the minimal semi-resolution of the relative log canonical model is crepant. It follows that $f^*(\omega_C \otimes \LL) \cong \omega_X$ away from type $\mathrm{II},\mathrm{III}$, $\mathrm{IV}$, and $\mathrm{N}_1$ fibers contained in $\Supp(F)$. 

We compute the contributions of these fiber types explicitly. In the minimal log resolution $Y$, the fibers consist of components $E,A$, and $D_i$, where $A$ is a reduced fiber intersecting the section $S$, and the components $D_i$ and $S$ are disjoint, each intersecting $E$ transversely. Note that $E$ and $D_i$ may support nonreduced fiber components.  We have a diagram
$$
\xymatrix{ & Y \ar[rd]^q \ar[ld]_p & \\ Z \ar[rd]_g & & X \ar[ld]^{f} \\ & C & }
$$
where $X$ is the log canonical model over $C$, $Z$ is the Weierstrass model, and $Y$ is the minimal log resolution obtained by finitely many blowups. In each case we have $p^*\omega_X = \omega_Y \otimes \calO_Y(B)$ where $B$ is effective and $p$-exceptional. 

Since the formula holds for Weierstrass models, we need to consider the other fibers of type $\mathrm{II}, \mathrm{III}$, $\mathrm{IV}$, and $\mathrm{N}_1$ appearing in $X$. These are obtained either by contracting $D_i$ in $Y$, or by contracting $D_i$ and $A$. First consider when $q$ contracts $D_i$. Since the $D_i$ are rational curves with negative self intesection on a smooth surface $Y$, the singularities of $(X, 0)$ are log canonical. In particular, by Lemma \ref{logcanonical} we have 
$$
q_*(\omega_Y(\sum D_i)) = \omega_X. 
$$
On the other hand,
$$
\omega_Y = p^*g^*(\omega_C \otimes \LL) \otimes \calO_Y(B) = q^*f^*(\omega_C \otimes \LL) \otimes \calO_Y(B).
$$
Therefore, by the projection formula:
$$
\omega_X = q_*(q^*f^*(\omega_C \otimes \LL) \otimes \calO_Y(B + \sum D_i)) = f^*(\omega_C \otimes \LL) \otimes q_* \calO_Y(B + \sum D_i).
$$
\indent Now $\calO_Y(B + \sum D_i)$ is effective and isomorphic to $\calO_Y$ away from $E \cup \big(\cup D_i\big)$. Since $q_*\calO_Y = \calO_X$, it follows that $q_*\calO_Y(B + \sum D_i) = \calO_X(\Delta)$ where $\Delta$ is effective and supported on $q(E \cup \big(\cup D_i\big))$.  The same argument works when $q$ contracts the $D_i$ and $A$. \\

Now we compute the contribution to $\Delta$ from each type of fiber. This is a local question in the neighborhood of such a fiber. Let $\varphi : Y \to Y'$ be the contraction of the component $A$ meeting the section induced by the transition from intermediate to twisted fiber in the relative log canonical model. Let $E$ denote the divisor supporting the nonreduced component of the intermediate fiber of $Y$ and denote $\varphi_*D := D'$ for any divisor $D$ on $Y$. 

As above, $A^2 = -n$ for $n = 6,4,3, 2$ for fibers of type $\mathrm{II}, \mathrm{III}$, $\mathrm{IV}$, or $\mathrm{N}_1$ respectively. Then
$$
\varphi^*K_{Y'} = K_Y + \frac{n - 2}{n}A.
$$
Furthermore, by the above, we know that 
$$
K_{Y'} = (f')^*(K_C + \LL) + \alpha E'
$$
for some $\alpha$. Here $f' : Y' \to C$ and $f : Y \to C$ are the corresponding elliptic fibrations. Then
$$
\varphi^*((f')^*(K_C + \LL)) + \alpha \varphi^*(E') = f^*(K_C + \LL) + \alpha E + \frac{\alpha}{n}A
$$
Since 
$\varphi_*K_Y = K_{Y'} 
$
as divisors and $K_Y - f^*(K_C + \LL)$ is supported on $E$, we see that 
$$
K_Y = f^*(K_C + \LL) + \alpha E
$$
and by equating the two expressions for $\varphi^*(K_{Y'})$ we get
$$
\frac{\alpha}{n} = \frac{n - 2}{n}
$$
so $\alpha = n - 2$.
\end{proof}

\begin{remark} For a type $\mathrm{N}_1$ fiber, $\alpha = 0$ so $\mathrm{N}_1$ fibers don't contribute to $\omega_X$ \emph{a posteriori}. \end{remark}

\begin{remark} When $X$ is normal and the coefficients $a_i$ are generic so that $(X, F_\calA)$ is klt, one can deduce that there exists an expression for the canonical bundle with form as in Theorem \ref{canonical} from Formula (3) in \cite[page 176]{fujinomori}. However, we need to control the precise support of the correction terms and consider semi-log canonical pairs. It is reasonable to expect that there is a generalization of \cite{fujinomori} for semi-log canonical pairs. 
\end{remark}

Next we describe how the log canonical divisor intersects the section:

\begin{prop}\label{prop:adjunction}   Let $(f: X \to C, S + F_\calA)$ be an $\calA$-weighted slc elliptic surface that is stable over $C$. Then $$(K_X + S +  F_\calA).S = 2g - 2 + \sum_i a_i.$$  
 \end{prop}

\begin{proof} Let $I \subset \{1 ,\ldots, n\}$ be the indices such that $a_i = 1$, and let $J$ be the complement of $I$. The section passses through the smooth locus of the surface in a neighborhood of  any fiber that is not marked with coefficient $a_i = 1$. This includes $F_j$ for $j \in J$. Therefore this formula follows from the adjunction formula away from $F_i$ for $i \in I$. On the other hand, for the twisted fibers $F_i$, this is the content of Proposition 4.3.2 of \cite{ln}. \end{proof}

\begin{definition}\cite{has}\label{hassettcurve} Let $g \in \Z_{\geq 0}$ and $\calA = (a_1, \cdots, a_n) \in \Q^n$ be such that $0 < a_i \leq 1$ and $2g -2 + \sum a_i >0$. An $\calA$-weighted stable pointed curve is a pair $(C, D_\calA = \sum a_i p_i)$ such that $C$ is a nodal curve of genus $g$, the $p_i$ are in the smooth locus of $C$, and $\omega_C(D_\calA)$ is ample.   \end{definition} 

\begin{cor} If $(f : X \to C, S + F_\calA)$ is an $\calA$-weighted stable elliptic surface, then $(C, \sum a_i p_i)$ is an $\calA$-weighted stable pointed curve where $p_i = f_*F_i$. \end{cor}

Let $(f : X \to C, S + F_\calA)$ be an slc elliptic surface such that $(C, \sum a_i p_i)$ is a weighted stable pointed curve where $p_i = f_*F_i$. In light of the above, we have that the log canonical model of $(X, S + F_\calA)$ is the same as the log canonical model of $(X, S + F_\calA)$ relative to $C$:

\begin{cor} Let $(f : X \to C, S + F_\calA)$ be a relatively stable elliptic surface such that $(C, \sum a_i p_i)$ is a weighted stable curve. Then $(X, S + F_\calA)$ is stable. In particular, $(X, S + F_\calA)$ is of log general type and its log canonical model is an elliptic surface. \end{cor} 

We are left to consider the following: 

\begin{cor}\label{cor:p1} The log minimal model program contracts the section of an $\calA$-weighted slc elliptic surface if and only if either
\begin{enumerate}[label = (\alph*)]
\item $C \cong \mb{P}^1$ and $\sum a_i \le 2$, or
\item $C$ is a genus $1$ curve and $a_i = 0$ for all $i$. 
\end{enumerate} 
\end{cor} 

In either of the two cases above, if $X = E \times C$ is a product then the contraction of the section $S$ is the projection $X \to E$ resulting in an elliptic curve as the log canonical model. Otherwise, the contraction of the section is birational and we obtain a pseudoelliptic. 

In case $(a)$, if the pair is of log general type then the resulting pseudoelliptic is the log canonical model. However, it is possible that that the pair is not of log general type in which case the log minimal model program will continue with either an extremal or log canonical contraction to produce a curve or point. In the next section, we describe how to determine the coefficients for which this happens. This is also discussed in greater detail in \cite{master}. 

In case $(b)$, the contraction of the section is necessarily the log canonical contraction and the resulting pseudoelliptic is the log canonical model. 

\section{Base curve of genus 0}\label{sec:genus0}

In the last section we arrived at the log canonical model of an $\calA$-weighted elliptic surface whenever the base curve has genus $g \geq 1$. We are left to analyze genus $0$ base curve case.

\begin{prop}\label{properlyelliptic} Let $f: X \to C$ be a properly elliptic surface with section $S$. Then $K_X + S$ is big. In particular, any $\calA$-weighted slc properly elliptic surface is of log general type. 
\end{prop}
\begin{proof} By assumption, $K_X = G + E$ where $G$ is an effective sum of fibers and $E$ is an effective divisor supported on fibral components. Then $(K_X + S).G > 1$ so $K_X + S$ is $f$-big and for a generic horizontal divisor $D$, the intersection $(K_X + S).D > 0$. It follows that $K_X + S + F_\calA$ is big for any $F_\calA$.  \end{proof} 

\begin{cor} Let $(f : X \to C, S)$ be a properly elliptic surface over $\mb{P}^1$. Then the log canonical model of $(f : X \to C, S + F_\calA)$ for any choice of marked fibers $F_\calA$ is either 
\begin{enumerate}[label = (\alph*)]
\item the relative log canonical model over $C$, or
\item the pseudoelliptic formed by contracting the section of the relative log canonical model. 
\end{enumerate} 
\end{cor}

This leaves $\deg \LL = 1,2$. Note that if the generic fiber of $f : X \to \mb{P}^1$ is smooth, then $\deg \LL = 1,2$ are exactly the cases corresponding to $X$ being rational ($\deg \LL = 1$) or birational to a K3 surface ($\deg \LL =2$). 

\begin{prop}\label{K3} Let $(f: X \to \mb{P}^1, S + F_\calA)$ be an $\calA$-weighted slc elliptic surface with section and marked fibers and suppose $\deg \LL = 2$.

\begin{enumerate}[label = (\alph*)]
\item If $\calA > 0$, then $K_X + S + F_\calA$ is big and the log canonical model is the pseudoelliptic obtained by contracting the section of the relative log canonical model;
\item If $\calA = 0$, then the minimal model program results in a pseudoelliptic surface and the log canonical contraction contracts this surface to a point. 
\end{enumerate} 
\end{prop}
\begin{proof} \begin{enumerate}[label = (\alph*)]
\item As a big divisor plus an effective divisor is big, it suffices to prove the result for $\calA = (\epsilon, \ldots, \epsilon)$ for some $0 < \epsilon \ll 1$. In this case, each type $II, III$ and $\mathrm{IV}$ fiber in the relative log canonical model $(g : Y \to \mb{P}^1, S + F_\calA)$ is a Weierstrass model. Then $\omega_{Y} = g^*(\omega_{\mb{P}^1} \otimes \LL)$ by the canonical bundle formula, but $\omega_{\mb{P}^1} \otimes \LL = \calO_{\mb{P}^1}$ since $\LL$ is degree $2$. Therefore 
$$
K_Y + S + F_\calA = S + \epsilon\left(\sum F_i\right)
$$
and the result follows as in Proposition \ref{properlyelliptic}. 

\item If $F_\calA = 0$ then the relative log canonical model is the Weierstrass model $(g : Y \to \mb{P}^1, S)$ and $K_Y = 0$ as in part $(a)$ so $K_Y + S = S$. We have $S^2 = -2$ by the adjunction formula so there is an extremal contraction of $S$ to obtain a pseudoelliptic $\mu : Y \to Y_0$ and $\mu_*(K_Y + S) = K_{Y_0} \sim_\Q 0$. Therefore $|mK_{Y_0}|$ is basepoint free and induces a log canonical contraction to a point. \end{enumerate} \end{proof} 

\begin{prop}\label{rational} Let $(X, F_\calA)$ be an $\calA$-weighted slc pseudoelliptic surface with marked fibers $F_\calA$. Denote by $Y$ the corresponding elliptic surface and $\mu : Y \to X$ the contraction of the section. Suppose $\deg \LL = 1$ and $0 < \calA \le 1$ such that $K_X + F_\calA$ is a nef and $\Q$-Cartier $\Q$-divisor. Then either 

\begin{enumerate}[label = \roman*)]
\item $K_X + F_\calA$ is big and the log canonical model is an elliptic or pseudoelliptic surface;
\item $K_X + F_\calA \sim_\Q \mu_*\Sigma$ where $\Sigma$ is a multisection of $Y$ and the log canonical map contracts $X$ onto a rational curve;
\item $K_X + F_\calA \sim_{\Q} 0$ and the log canonical map contracts $X$ onto a point. 
\end{enumerate} 

\noindent The cases above correspond to $K_X + F_\calA$ having Iitaka dimension $2,1$ and $0$ respectively. 

\end{prop} 

\begin{proof} By the abundance in dimension two (Proposition \ref{prop:abundance} and \cite[Theorem 1.4]{hx1} \cite[Theorem 1.4]{fg}), we know that $K_X + F_\calA$ is semiample. Let $\varphi: X \to Z$ be the Iitaka fibration. If $\varphi$ is birational, we are in situation $(i)$ and $\kappa(X, K_X + F_\calA) = 2$. Thus suppose $\varphi$ is not birational.

Let $f: Y \to C$ be the elliptic fibration whose section $S$ is contracted to obtain $X$ and let $\mu :Y \to X$ be this contraction. Consider $g = \varphi \circ \mu : Y \to Z$. Let $G$ be a generic fiber of $f$. Then $G^2 = G.B = 0$ for any fiber component $B$ of the elliptic fibration. Writing
$$
\mu^*(K_X + F_\calA) = K_Y + tS + \widetilde{F}_\calA
$$
where $\widetilde{F}_\calA$ is the strict transform of $F_\calA$, we have that 
$$
\mu^*(K_X + F_\calA).G = t.
$$
On the other hand, 
$$
\mu^*(K_X + F_\calA).G = (K_X + F_\calA).\mu_*G \geq 0
$$
by the projection formula and the assumption that $K_X + F_\calA$ is nef. 

Suppose $t = 0$ so that $\mu^*(K_X + F_\calA).G = 0$ for a general fiber $G$. It follows that $(K_Y + tS + \widetilde{F}_\calA).B = 0$ for \emph{all} fiber components $B$. Indeed in the case of a Weierstrass of twisted fiber there is a unique fiber component  $B$, and $dB \sim_\Q G$ for some $d \geq 1$. For an intermediate fiber consisting of a reduced component $A$ and a component $E$ supporting a nonreduced component, we have that $A + dE \sim_\Q G$ for some $d \geq 2$ so
$$
(A + dE).\mu^*(K_X + F_\calA) = 0
$$
but $K_X + F_\calA$ is nef so $A.\mu^*(K_X + F_\calA) = E.\mu^*(K_X + F_\calA) = 0$. 

Therefore $\mu^*(K_X + F_\calA) = K_Y + \widetilde{F}_\calA$ is trivial on both the fibers and the section and so must be numerically trivial. By abundance, it must be rationally equivalent to $0$. Therefore,
$
K_X + F_\calA \sim_\Q 0
$
so we are in case $(iii)$ and $\varphi : X \to Z$ is the contraction to a point. On the other hand, if $\varphi : X \to Z$ is the contraction to a point, then it is immediate that $K_X + F_\calA \sim_\Q 0$ so that we are in case $(iii)$ if and only if $t = 0$. 

This leaves only the case where $t > 0$ and $\varphi : X \to Z$ is a contraction onto a curve. Note that $Z$ is necessarily rational since the normalization of $X$ is a rational surface. Now $\mu^*(K_X + F_\calA)$ is ample on the generic fiber of $f : Y \to C$ and $K_X + F_\calA$ is base point free so it is linearly equivalent to an effective nonzero divisor $D$ that avoids the point $\mu(S)$. Therefore $\mu^*(K_X + F_\calA)$ is linearly equivalent to an effective horizontal divisor. That is, $\mu^*(K_X + F_\calA) \sim_\Q \Sigma$ where $\Sigma$ is an effective multisection and $K_X + F_\calA \sim_\Q \mu_*\Sigma$ since $\Sigma$ is contained in the locus where $\mu$ is an isomorphism. \end{proof}

\begin{remark}\label{rem:contract} The proposition above then gives us a method for determining which situation of $(i), (ii),$ and $(iii)$ we are in. Indeed since $K_X + F_\calA$ is nef, it is big if and only if $(K_X + F_\calA)^2 > 0$. Furthermore,
$
K_X + F_\calA \sim_\Q 0
$
if and only if $t = 0$. Thus when $K_X + F_\calA$ is not big, it suffices to compute whether $t > 0$ or not to decide whether the log canonical contraction morphism contracts the pseudoelliptic to a curve or a point.  \end{remark}

\section{Wall and chamber structure} \label{sec:calculationsection}
In this section we briefly discuss the wall and chamber in the domain of weights $\calA$ for an $\calA$-weighted slc elliptic surface $(f : X \to C, S + F_\calA)$. By the results in the rest of the paper the log canonical model remains the same within each chamber and changes across each wall. We use these walls in \cite{master} to describe how compactifications of the moduli space of $\calA$-weighted stable elliptic surfaces vary as the weight vector $\calA$ varies. 

Finally we end with a detailed example of a rational elliptic surface to demonstrate the various transitions the log canonical model undergoes across each type of wall. 

\subsection{Transitions from twisted to Weierstrass form} First we note the weights for which the relative log canonical models change as the weight decreases from $1$ to $0$. There are thus walls: 
\begin{itemize}
	\item At $a_i = 1$ where a non-stable fiber transitions between twisted and intermediate inside the chamber $a_i = 1 - \epsilon$ for $0 < \epsilon \ll 1$.
	\item At $a_i = 5/6$ where a type $\mathrm{II}$ fiber transitions between intermediate and Weierstrass.
	\item At $a_i = 3/4$ where a type $\mathrm{III}$ fiber transitions between intermediate and Weierstrass.
	\item At $a_i = 2/3$ where a type $\mathrm{IV}$ fiber transitions between intermediate and Weierstrass.
	\item At $a_i = 1/2$ where a type $\mathrm{N}_1$ transitions between intermediate and Weierstrass.
	\item At $a_i = 0$ where a non-stable fiber that is not of the above form transitions from intermediate to Weierstrass.
\end{itemize}
Across each of these walls, the relative log canonical model exhibits a birational transformation.  

\subsection{Contraction of the section} By Proposition \ref{prop:adjunction}, there is a wall at $2g(C) - 2 + \sum a_i = 0$ where the section is contracted by the log canonical contraction. In the chambers below the wall the section is contracted by an extremal contraction. This contraction is birational except in the case when $X$ is birational to a product $E \times C$ in which case it is the projection to $E$. Note that this wall only exists when $g = 0$ or when both $g = 1$ and $\calA = 0$.

\subsection{Pseudoelliptic Contractions}\label{sec:typeIII}
These transitions occur when a pseudoelliptic surface $X$ is contracted to a rational curve or a point. Let $(f : Y \to \mb{P}^1, S + F_\calA)$ be the corresponding elliptic surface with $\mu : Y \to X$ the contraction of the section. By Proposition \ref{properlyelliptic}, these walls do \emph{not} occur for $\deg \LL \geq 3$. When $\deg \LL = 2$, there is a single such wall at $\calA = 0$ when the log canonical contraction contracts $X$ to a point by Proposition \ref{K3}. 

For $\deg \LL = 1$, Proposition \ref{rational} guarantees that there are possibly two such walls. The first is when $(K_X + \mu_*F_\calA)^2 = 0$, so that $(X, F_\calA)$ is not of log general type. If $K_X + \mu_*F_\calA \sim_\Q 0$ then the log canonical contraction maps to a point. If $K_X + \mu_*F_\calA \not \sim_\Q 0$, the log canonical contraction maps to a rational curve and there is a further wall when $t = (K_X + \mu_*F_\calA).\mu_*G = 0$, where $G$ is a general fiber of $f$. At this wall the log canonical contraction maps to a point.

These walls are less explicit in that they depend on the particular configurations of singular fibers that are marked. However, since there are only finitely many combinations of singular fibers on $Y$ with $\deg \LL = 1$, one may compute these walls explicitly in any particular case as is illustrated by the following Lemma and the example in the next subsection. 

\begin{lemma} In the situation above, suppose $a_i < 1$ for all $i$. Then there is a wall at $\sum a_i = 1$ where the log canonical contraction maps to a point. \end{lemma} 

\begin{proof} Without loss of generality we take $f: Y\to \mb{P}^1$ to be the relative log canonical model. Since $a_i < 1$ for all $i$, the surface $Y$ has no twisted fibers and so $S$ passes through the smooth locus of $f$ and $S^2 = -1$. Therefore $\mu : Y \to X$ is the contraction of a $(-1)$ curve and we compute $$\mu^*(K_X + \mu_*F_\calA) = K_Y + \left(\sum a_i - 1 \right)S + F_\calA.
$$
Thus the coefficient of $S$ is $0$ precisely when $\sum a_i = 1$ so the result follows by Proposition \ref{rational}. \end{proof}

\subsection{A rational example} \label{example} Let $X \to \PP^1$ be a rational elliptic surface that contains exactly two singular fibers of type $I_0^*$ whose existence follows from Persson's classification of rational elliptic surfaces \cite{persson}. Denote the reduced singular fibers by $F_0$ and $F_1$. All other fibers are smooth and we denote the class of a general fiber by $G$. We fix $F_1$ to have coefficient $1$ and give $F_0$ and $G$ the same coefficient $\alpha$. Then $\calA = (\alpha,\alpha,1)$ and we have the pair
$$
(X \to C, S + \alpha(G + F_0) + F_1)
$$
Since $F_1$ is kept with coefficient $1$ it is a twisted model for all $\alpha$. Thus $F_1 \sim_\Q 1/2G$. Furthermore, by the canonical bundle formula, $K_X \sim_\Q - G$ since $\deg \LL = 1$ and there are no fibers of type $II, III$, or $\mathrm{IV}$. Putting this together (with $F_\calA = \alpha (G + F_0) + F_1)$, we have
$$
K_X + S + F_\calA = S + \alpha F_0 + (\alpha - 1/2)G.
$$
When $\alpha = 1$, the log canonical model is  $f : X_1 \to \mb{P}^1$ with two twisted $I_0^*$ fibers. 

For $1/2 < \alpha < 1$, $F_0$ becomes an intermediate fiber with components $A$ and $E$ in the relative log canonical model. The log canonical model is an elliptic surface $X_{1 - \epsilon}$ with a map $X_{1 - \epsilon} \to X_1$ contracting $A$. Next we check
$$
(K_X + S + \alpha(G+ F_0) + F_1).S = 2\alpha - 1.
$$
At $\alpha = 1/2$ the section of $X_1$ is contracted to a pseudoelliptic by the log canonical contraction $\mu : X_{1 - \epsilon} \to X_{1/2}$. 

For $\alpha < 1/2$ the map $\mu : X_{1 - \epsilon} \to X_{1/2}$ is an extremal contraction. Writing
$$
\mu^*(K_{X_{1/2}} + \mu_*(F_\calA)) = K_{X_{1 - \epsilon}} + tS + F_A
$$
we compute $t = 4\alpha - 1$ by intersecting both sides with $S$ and using that $S^2 = -1/2$ since $S$ passes through an $A_1$ singularity along the twisted $I_0^*$ fiber. Furthermore, using $F_0^2 = -1/2$ for an intermediate $\mathrm{I}_0^*$ fiber,
\begin{align*}
(K_{X_{1 - \epsilon}} + tS + F_A)^2 	&= ((4\alpha - 1)S + \alpha F_0 + (\alpha - 1/2)G)^2 \\ 
					&= (4\alpha - 1)^2(-1/2) + 2(4\alpha - 1)(2\alpha - 1/2) - \alpha^2/2 \\ 
					&= \frac{1}{2}[(4\alpha - 1)^2 - \alpha^2] = \frac{1}{2}(3\alpha - 1)(5\alpha - 1).
\end{align*}

Therefore there is a pseudoelliptic contraction at $\alpha = 1/3$ where $K_{X_{1/2}} + \mu_*(F_\calA)$ is no longer big. Since $t > 0$ for $1/4 < \alpha \le 1/3$, the log canonical class is a multisection and the log canonical contraction maps onto a rational curve. Finally at $\alpha = 1/4$, $t = 0$ so the log canonical class is trivial and the log canonical contraction maps to a point.

\begin{table}[!h]\caption{\footnotesize{We show the transformation of the elliptic surface $X \to \PP^1$ as we lower the weight $\alpha$ on $F_0$ and $G$. We always keep $F_1$ with a fixed weight 1.}}
 \begin{tabular}{|c|c|c|c|c|} \hline $0 \le \alpha \le 1/4$ & $1/4 < \alpha \le 1/3$ & $1/3 < \alpha \le 1/2$ & $1/2 <\alpha < 1, \textrm{elliptic}$ & $\alpha = 1$, elliptic \\ pt & \text{curve} & \text{pseudoelliptic} & \text{$F_0$ intermediate} & \text{$F_0$ twisted} \\     \tikz\draw[fill=black] (0,0) circle (.35ex); &   \includegraphics[scale=.6]{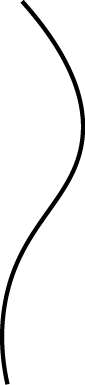} &   \includegraphics[scale=.80]{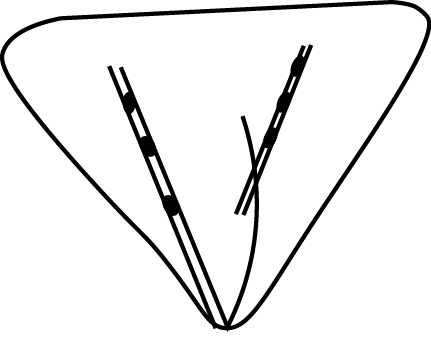} &   \includegraphics[scale=.82]{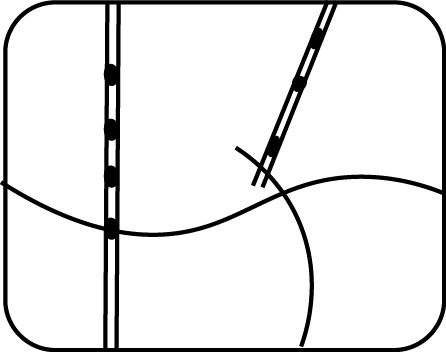} &   \includegraphics[scale=.82]{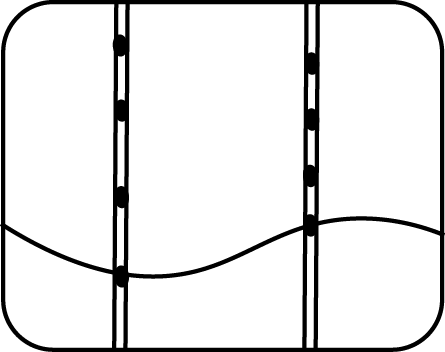} \\ \hline  \end{tabular}\end{table}

\section*{\refname}

\bibliographystyle{alpha}
\bibliography{master}
\end{document}